\documentclass[a4paper, 12pt]{article}
\usepackage{graphicx}
\usepackage[british]{babel}

\usepackage[dvipsnames]{xcolor}

\usepackage{hyperref} 
\hypersetup{
    colorlinks=true,
    linkcolor=Red,
    filecolor=magenta,      
    urlcolor=magenta,
}

\usepackage{verbatim}
\usepackage{amsmath, amsthm, amsfonts, amsbsy, amssymb, scrextend}
\usepackage{color}
\usepackage{tikz}
\usepackage[export]{adjustbox}
\usepackage{float}

\newtheorem{theorem}{Theorem}[section]

\newtheorem{definition}[theorem]{Definition}
\newtheorem{lemma}[theorem]{Lemma}
\newtheorem{corollary}[theorem]{Corollary}
\newtheorem{conjecture}[theorem]{Conjecture}
\newtheorem{proposition}[theorem]{Proposition}
\newtheorem{remark}[theorem]{Remark}
\newtheorem{example}[theorem]{Example}

\newcommand{\by}{{\bf y}}
\newcommand{\bx}{{\bf x}}

\newcommand{\be}{{\bf e}}
\newcommand{\bu}{{\bf u}}

\def\dir{\curvearrowright} 
\def\dirl{\rightsquigarrow} 

\newcommand{\T}{{\sf T}\,}

\def\|{{\,|\, }}
\def\d{{\,\sf d}}
\newcommand{\bv}{{\bf v}}

\newcommand{\I}{{\bf I}}
\newcommand{\ts}{{\tilde\sigma}}

\newcommand{\F}{{\mathcal{F}}}
\newcommand{\Z}{\mathbb{Z}}
\newcommand{\R}{\mathbb{R}}
\newcommand{\C}{\mathbb{C}}

\newcommand{\eps}{\varepsilon}

\def\a{\alpha}
\def\b{\beta}

\def\l{\lambda}
\renewcommand{\P}{\mathbb{P}}
\newcommand{\E}{\mathbb{E}}

\newcommand{\A}{{\bf A}}

\begin{document}

\title{Linear competition processes and generalized P\'olya urns with removals}

\author{
Serguei Popov\footnote{Centro de Matem\'atica, University of 
Porto, Porto, Portugal.
Email address: serguei.popov@fc.up.pt}
\and 
Vadim Shcherbakov\footnote{Department of Mathematics, Royal Holloway, University of London, UK. Email address: vadim.shcherbakov@rhul.ac.uk}
\and Stanislav Volkov\footnote{Centre for Mathematical Sciences, Lund University, Sweden. Email address: stanislav.volkov@matstat.lu.se}~\footnote{Research supported by Crafoord grant no.~20190667 and Swedish Research Council grant VR~201904173}
}
\maketitle

\begin{abstract} 
A competition process is a continuous time Markov chain that can be interpreted as a system of interacting birth-and-death processes, the components of which evolve subject to a competitive interaction. This paper is devoted to the study of the long-term behaviour of such a competition process, where a component of the process increases with a linear birth rate and decreases with a rate given by a linear function of other components. A zero is an absorbing state for each component, that is, when a component becomes zero, it stays zero forever (and we say that this component becomes extinct). We show that, with probability one, eventually only a random subset of non-interacting components of the process survives. A similar result also holds for the relevant generalized P\'olya urn model with removals.
\end{abstract}



%\tableofcontents

\section{Introduction}
A classical birth-and-death process on the set of non-negative integers is a continuous time Markov chain (CTMC) which evolves as follows. When the process is at state $k$, 
it can jump either to state $k+1$, or to state $k-1$ (provided $k>0$), with transition rates that are state-dependent. The long term behaviour of the birth-and-death process is well studied. Given a set of transition rates one can, in principle, determine whether the corresponding CTMC is (positive) recurrent or (explosive) transient, and compute various other characteristics of the process.

A multivariate birth-and-death process is a CTMC with values in a multi-dimensional non-negative orthant, and the dynamics of which is similar to that of the classical birth-and-death process. A multivariate birth-and-death process can often be interpreted as a {\em system of interacting} one-dimensional birth-and-death processes. A competition process is, probably, the most known example of such Markov chains. For instance, competition processes with non-linear interaction (e.g., of the Lotka-Volterra type) were originally proposed to model competition between populations (please see \cite{Anderson}, \cite{Iglehart}, \cite{Reuter}, \cite{Renshaw} and references therein). 
 
In contrast to the one-dimensional case, the long term behaviour of multivariate birth-and-death processes is much less known, even though results are available in some special cases. While we do not provide a complete review of the relevant literature, we would like to mention the papers~\cite{Hutton}, \cite{Kesten1972} \cite{Kesten1976}, and \cite{Mode}, in addition to the references above, where the technical framework is somewhat close to that of the present paper. The approach to studying a multivariate birth-and-death process depends on a particular model. For example, it is well known that reversibility greatly facilitates the study of the long term behaviour of the birth-and-death process (e.g. see~\cite{Karlin}). This is also the case in the multivariate situation (\cite{JShV19} and~\cite{ShV15}). On the other hand, in the non-reversible case the Lyapunov function method (\cite{MPW}) is widely used. The method has been applied to studying the long term behaviour of the multivariate birth-and-death processes since the 1960s (see \cite{Reuter}), in order to establish recurrence vs.\ transience, as well as to detect some more subtle phenomena (\cite{MS}, \cite{ShV19}).

In the current paper we study the long term behaviour of a {\it linear} competition process:
components increase as linear pure birth processes and decrease with a death rate, given by a linear function depending on other components. The functions determining death rates are, in turn, determined by a non-negative matrix, called {\em the interaction matrix.} When a component of the process becomes zero, it stays zero forever (becomes extinct); in other words, zero is an absorbing state for each component. If a component of the process never becomes zero, we say it {\em survives}.

The main result of the paper is the following. With probability one, eventually only a random subset of the components of the process survive. Every limit set of survivals is formed by mutually non-interacting components, so that the survived components evolve as independent linear pure birth processes (Yule processes). This result can be equivalently stated in terms of the discrete time Markov chain corresponding to the linear competition process (the embedded Markov chain). The embedded Markov chain can be regarded as an urn model with removals, where balls of several types are added to and removed from the urn with probabilities induced by the transition rates of the competition process. Hence, with probability one, eventually only balls of a random subset of types will be left in the urn (survive). The numbers of balls of the surviving  types will evolve according to the classical generalised P\'olya urn model. 

A crucial step in our proof is to show that, with probability one, eventually one of the interacting components becomes extinct. Showing this fact is straightforward, provided that the competition is sufficiently strong. This is similar to the models with non-linear competitive interaction, where strong interactions generate a sufficient drift directed towards the boundary. At the same time, more subtle phenomena, such as quasi-stationary distributions or extinction probabilities, are of primary interest in those models (e.g. see \cite{NikolayDenis},~\cite{Rachel},~ \cite{Meleard} and references therein). 

Showing extinction is much harder when the interaction is weak. It turns out that the phase transition in the strength of the interaction is determined by the largest eigenvalue of the interaction matrix. This fact is not at all surprising, since the dynamics of the linear competition process has a striking resemblance with that of multi-type branching processes (MTBP), where eigenvalues (the largest one, in particular) of the mean drift matrix play a crucial role. This similarity allows us to adopt the well-known method for studying both MTBPs and urn-related models (\cite{Athreya1},~\cite{Athreya2},~\cite{Svante2004}). In particular, the scalar products of eigenvectors of the interaction matrix and the embedded Markov chain provide us with useful semimartingales.

The rest of the paper is organised as follows. In Section~\ref{sec-main} we state the model and the results rigorously. In Section~\ref{sec-prelim} we prepare all necessary ingredients for the proof of the main results, which are given in Section~\ref{sec-proofs}. Section~\ref{sec-lemmas_proofs} contains the proofs of the lemmas, and in Section~\ref{sec-examples} we describe some interesting examples.

\section{The model and the main result}
\label{sec-main}
Let $\Z_{+}$ be the set of all non-negative integers, and let $\R_{+}$ be the set of all non-negative real numbers, both including zero. For a vector $\bx=(x_1,\dots,x_N)\in \R^N$ we will write $\bx>0$ whenever all $x_i>0$. A vector $\bx=(x_1,\dots,x_N)\in \R^N$ is understood as a column vector, so that $\bx^{\T}$ is a row vector. Further, $\bx\cdot\by=\bx^{\T}\by$ denotes a Euclidean scalar product of vectors $\bx$ and $\by$, and $1_{D}$ denotes an indicator of an event (or set) $D$. All random variables are realised on a certain probability space $(\Omega, \F, \P)$. The expectation with respect to the probability $\P$ will be denoted by $\E$. 
The real part and the imaginary part of a complex number $z$ will be denoted by 
$\Re(z)$ and  $\Im(z)$ respectively.

\begin{definition}
\label{D1}
Fix an integer $N\ge 1$. An $N\times N$ matrix $\A=(a_{ij})$ with non-negative elements and zeros on the main diagonal is called {\em an interaction matrix}. 
\end{definition}

Given a number $\alpha>0$ and an interaction matrix $\A=(a_{ij})$ consider a CTMC $X(t)=\left(X_{1}(t),\ldots, X_{N}(t)\right)\in\Z_{+}^{N}$, $t\in\R_{+}$, with the following transition rates 
\begin{equation}
\label{rates}
q_{\bx\by}=\begin{cases}
\alpha x_i,& \by=\bx+\be_i;\\
\left(\sum_{j=1}^N a_{ij} x_j\right)1_{\{x_i>0\}},& \by=\bx-\be_i,
\end{cases}
\end{equation}
where $\bx=(x_1,\ldots ,x_N)$, $\by\in \Z_{+}^{N}$, and $\be_i$ is the $i$-th 
unit vector in $\Z_{+}^{N}$, i.e. a vector such that its $i$-th component is 
equal to $1$ and all its other components are equal to $0$. In what follows, we 
refer to a CTMC $X(t)$ with transition rates~\eqref{rates} as a linear 
competition process (LCP).

\begin{remark}
{\rm The quantity $a_{ij}\ge 0$ indicates how much component~$i$ is affected by component~$j$. In biological terms, the fact that $a_{ij}>0$ can be interpreted as a predator~$j$ hunting prey~$i$. }
\end{remark}

\begin{remark}
{\rm If $\A= {\bf 0}$, then the LCP $X(t)$ is a collection of independent pure linear birth processes with parameter $\alpha$. The latter means that if a component is at state $k>0$, then it can only jump to state $k+1$ with rate $\alpha k$. Such a process is also known as Yule process (see e.g.~\cite{Karlin}). In general, CTMC $X(t)$ is a special case of the so called competition process (see the references above) and can be interpreted as a system of interacting birth-and-death processes with linear interaction. }
\end{remark}

Let $\zeta(n)=(\zeta_1(n),\ldots ,\zeta_{N}(n))\in\Z_{+}^{N}$, $n\in \Z_{+}$, be 
the embedded Markov chain (the embedded process) corresponding to the LCP 
$X(t)$. In other words, $\zeta(n)$ is a discrete time Markov chain (DTMC) with 
the following transition probabilities
\begin{equation}
\label{embedded}
\begin{split}
\P\left(\zeta(n+1)=\zeta(n)+\be_i\,\big|\F_n\right)&=\frac{\alpha \zeta_i(n)}{R(\zeta(n))},\\
\P\left(\zeta(n+1)=\zeta(n)-\be_i\,\big|\F_n\right)&=
\frac{\sum_{j=1}^N a_{ij}\zeta_j(n)}{R(\zeta(n))}\, 1_{\{\zeta_i(n)>0\}},
\end{split}
\end{equation}
where $\F_n$ is the natural filtration generated by $\zeta(k),\, k\leq n,$ and 
\begin{equation}
\label{R}
R(\zeta)=\sum_{i=1}^{N}\left(\alpha \zeta_i+ 1_{\{\zeta_i>0\}}\sum_{j=1}^N a_{ij} \zeta_j\right)\quad\text{for}\quad
\zeta=(\zeta_1,\ldots ,\zeta_N)\in\Z_{+}^N.
\end{equation}

\begin{remark}
{\rm 
Note that the DTMC $\zeta(n)=(\zeta_1(n),\ldots ,\zeta_{N}(n))$ can be regarded 
as an urn model with removals, where $\zeta_i(n)$ is interpreted as a number of 
balls of type $i$.
}
\end{remark}

Before we formulate the main theorem, we need to introduce a few definitions from the graph theory. Observe that the transposed interaction matrix $\A^\T$ can be regarded as a weighted adjacency matrix of a {\it directed} graph $G$ defined below.

\begin{definition}
\label{D2}
The graph $G=G(\A)$ corresponding to the interaction matrix $\A$ is a loopless 
directed graph~$G$ with the vertex set $V=\{1,\ldots ,N\}$, where vertices $i$ 
and $j$ are connected by a directed edge (written as $i\dir j$) if and only if 
$a_{ji}>0$.
\end{definition}

\begin{definition}
\label{Dpath}
Let $G=(V, E)$ be a directed graph with vertex set $V$ and edge set $E$.
\begin{enumerate}
\item We say that there is a directed path from $v\in V$ to $w\in V$ and write $v\dirl w$, if there exists a sequence of vertices $v=v_1$, $v_2$, $\dots$, $v_k=w$ of $G$ such that $v_{i}\dir v_{i+1}$ for $i=1,2,\dots,k-1$. 
\item We call a non-empty directed graph $G$ strongly connected if it either consists of just one vertex, or if any two distinct vertices~$v,w\in G$ satisfy $v\dirl w$ and $w\dirl v$. Equivalently, if $G=G(\A)$, then this is equivalent to the irreducibility of matrix $\A$, i.e.\ the matrix ${\bf I}+\A+\A^2+\dots+\A^n$ is strictly positive for some sufficiently large $n$.
\end{enumerate}
\end{definition}

\begin{definition}\label{D4}
Let $G=(V, E)$ be a directed graph with vertex set $V$ and edge set $E$.
\begin{enumerate}
\item Given a subset of vertices $V'\subset V$ the corresponding induced subgraph is graph $G' = (V', E')$ with edge set $E'$ inherited from graph $G$.
\item Let $G'\subset G$ be a subgraph induced by a non-empty subset of vertices $V'\subset V$. The subgraph $G'$ is called a source subgraph, if there are no $v\in V\setminus V'$ and $v'\in V'$ such that $v\dir v'$ (i.e., there are no edges incoming to $G'$).
\end{enumerate}
\end{definition}

\begin{remark}
{\rm If the directed graph $G$ is disconnected, then the corresponding linear competition process will behave independently on each of the connected components of $G$, with the transition rates appropriate for that component (of course, with a different sub-matrix of $\A$). Also, whenever one of the components of the process ($X_i$ or $\zeta_i$ respectively) becomes zero, this is equivalent to removing the corresponding vertex $i$ from the vertex set of $G$, along with all the edges incoming to or outgoing from $i$ (that is, crossing out simultaneously the $i$th row and the $i$th column from~$\A$). As a result, a connected component of $G$ containing vertex $i$ might split into more than one connected components. }
\end{remark}

Theorem~\ref{T1} below is the main result of the paper.

\begin{theorem}
\label{T1}
Let $X(t)=(X_1(t),\ldots ,X_N(t))\in\Z_{+}^N,\, t\in\R_{+},$ be a LCP with 
transition rates~\eqref{rates} specified by a parameter $\alpha>0$ and an 
interaction matrix $\A$. Let 
$\zeta(n)=(\zeta_1(n),\ldots ,\zeta_{N}(n))\in\Z_{+}^{N}$, \, $n\in \Z_{+}$, be 
the corresponding embedded DTMC with transition probabilities~\eqref{embedded}. 

Suppose that $X(0)=\zeta(0)>0$. Then, for every subset ${\cal 
I}=\{i_1,i_2,\dots, i_K\}\subseteq \{1,\ldots ,N\}$ such that $a_{ij}=a_{ji}=0$ 
for all $i, j\in {\cal I}$ and containing
{ at least one vertex from each strongly connected subgraph of $G(\A)$}
\begin{align*}
\lim_{t\to \infty}X_i(t)=\lim_{n\to \infty}\zeta_i(n)=
\begin{cases}
\infty, &\text{ if } i\in {\cal I};\\
0, &\text{ if }\, i\notin{\cal I}
\end{cases}
\end{align*}
with positive probability. 

No other limiting behaviour is possible. That is, with probability one, a random subset of non-interacting components of the process $X(t)$ survives, and the survived components behave as independent Yule processes with parameter $\alpha$. As a result, for large $n$ the process $\{\zeta_{i}(n),\, i\in {\cal I}\}$ has the same distribution as the classical P\'olya urn with $K$ different types of balls.
\end{theorem}

\begin{example}
{\rm Suppose that all non-diagonal elements of $\A$ are strictly positive, i.e.\  graph $G(\A)$ is a complete graph (every pair of the process components interact with each other). Then, by Theorem~\ref{T1}, only one population will survive a.s.}
\end{example}

\begin{example}
{\rm Consider a directed graph $G$ with eight vertices $1,2,\ldots ,8$ depicted 
in Figure~\eqref{figure1}. It follows from Theorem~\ref{T1} that the set of 
limit configurations of surviving components is determined by the following 
subsets of vertices $(1, 3, i)$, or $ (1, 5, i)$, or $ (1, 3, 5, i)$, for $ i= 
6, 7, 8$. For instance, the subset $\{1, 3,5, 6\}$ can be obtained as follows. 
First, vertex $2$ is removed with all incoming and outgoing edges from the graph 
(i.e., component $X_2$ becomes extinct). Then, say vertices $7, 8$ and $4$ are 
subsequently removed. It is easy to see that the same surviving subset can be 
obtained in many ways. Note that the directed graph $G$ is not strongly 
connected, e.g. there is no path connecting vertex $2$ and vertex $1$. There are 
two strongly connected source subgraphs in this graph: a single vertex $\{1\}$ 
(source vertex), and the subgraph induced by vertices $\{6,7,8\}$. }
\end{example}

\begin{figure}[H]
 \centering
 \includegraphics[scale=0.9]{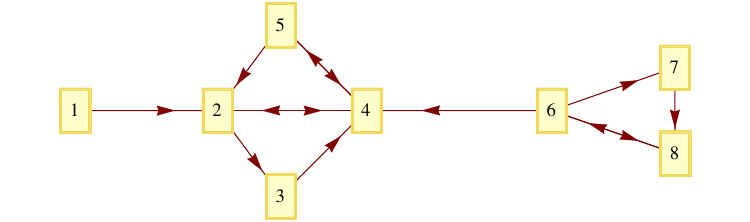}
 \vspace*{-4mm}
\caption{\small{Graph with 8 vertices }}
 \label{figure1}
\end{figure}

Finally, we describe a relevant urn model with $N$ different types of balls. For 
simplicity, assume that {\it both} $\alpha$ and {\it all} $a_{ij}$ are integers. 
Consider a DTMC $Y(n)=(Y_1(n),\ldots ,Y_N(n))\in \Z_{+}^N$,\, $n\in\Z_{+},$ 
where $Y_i(n)$ represents a number of balls of type $i=1,\ldots ,N$ in a urn. 
The dynamics of the model is as follows. Suppose an urn contains $Y_i\ge 1$ 
balls of type $i\in\{1,2,\dots,N\}$. Pick a ball of type $i$ with probability 
proportional to $Y_i$, and then return it to the urn with $\alpha$ additional 
balls of type $i$; at the same time for each $j\neq i$ {\em remove} $\tilde 
a_{ji}(n):=\min\{a_{ji},Y_j\}$ balls of type $j$. By doing so, we obtain a 
generalized P\'olya urn model with {\em removals}.

Formally, the transition probabilities of the urn are given by
\begin{equation}
\P\left(Y(n+1)=Y(n)+\alpha\be_i-\sum_{j=1}^N\tilde a_{ji}(n)\be_j\, \Big|\, Y(n)\right)=
\frac{Y_i(n)}{\sum_{j=1}^NY_j(n)},\quad i=1,\ldots ,N.
\end{equation}
Such a model with $\alpha=0$ and $A=\begin{pmatrix}0 &1\\1 &0 \end{pmatrix}$, called {\em the OK Corral model}, was considered in~\cite{Kingman}. Another similar model with removals, called {\em Simple Harmonic Urn}, was studied in~\cite{Crane}. The connection between the above urn model and the LCP is explained in Section~\ref{ProofT2}. Our results for the LCP extend to the urn model as follows.
\begin{theorem}
\label{T2}
The statement of Theorem~\ref{T1} for the DTMC $\zeta(n)=(\zeta_1(n),\ldots 
,\zeta_N(n))\in \Z_{+}^N,$ $n\in\Z_{+}$,
 holds also for the urn process $Y(n)=(Y_1(n),\ldots ,Y_N(n))\in \Z_{+}^N$.
\end{theorem}

\section{Preliminaries}\label{sec-prelim}

\subsection{The model graph}\label{graph}

Recall Definitions~\ref{Dpath} and \ref{D4}.
\begin{lemma}
\label{lemgraphtheory}
Any non-empty directed graph $G=(V, E)$ contains a strongly connected source subgraph. 
\end{lemma}
\begin{proof}[Proof of Lemma~\ref{lemgraphtheory}]
Let us introduce an equivalence relation $ \simeq$: we say that $v \simeq w$ if  $v\dirl w$ and  $w\dirl v$, with the convention that $v\simeq v$. It's trivial to check that the usual properties of reflexivity, symmetry, and transitivity are fulfilled. 

This equivalence relation provides a partition of the underlying set into disjoint equivalence classes; let us denote the class containing vertex $v$ by $V(v)$.
Therefore, we can partition the vertex set $V$ of the graph as follows
$$
V=V(v_1)\sqcup V(v_2)\sqcup\dots \sqcup V(v_l)
$$
for some vertices $v_1, v_2,\dots, v_l$.

Consider a directed graph $\tilde G$ with $l$ vertices $\tilde v_1,\ldots,\tilde v_l$, where $\tilde v_i\dir \tilde v_j$, whenever there are $v\in 
V(v_i)$ and $w\in V(v_j)$ such that $v\dir w$. The graph $\tilde G$ cannot have 
cycles. Indeed, if $\tilde v_1\dir \tilde v_2\dir \dots\dir\tilde v_m\dir\tilde 
v_1$ for some vertices $\tilde v_1, \tilde v_2,\dots, \tilde v_m$, then all 
vertices of $\bigcup_{k=1}^m V(v_k)$ must belong to the {\em same} $V(v)$ for 
some vertex $v$, leading to a contradiction.

Since $\tilde G$ does not have cycles, it is a tree or a forest (in case it is not connected). It is also finite, hence it must have at least one root, i.e.\ a vertex $\tilde v_i$ for which there is no $j$ such that $\tilde v_j\dir \tilde v_i$. Then the subgraph of $G$ induced by the set of vertices is $V(v_i)$ is indeed a strongly connected source subgraph.
\end{proof}

\begin{lemma}
\label{subgr}
Let $G'\subseteq G$ be a source subgraph of $G$ induced by a subset of vertices 
$V'=\{i_1,\ldots ,i_{N'}\}\subseteq\{1,\ldots ,N\}$, where $2\leq N'\leq N$. 
Let $X'(t)=(X_{i}(t),\, i\in {\cal I})\in\Z_{+}^{N'}$, $t\in\R_{+}$i.e. $X'(t)$ 
is a restriction of the LCP $X(t)$ on subgraph $G'$. Then the random process 
$X'(t)\in \Z_{+}^{N'}$ is a LCP with transition rates~\eqref{rates} specified by 
parameter $\alpha$ and interaction $N'\times N'$ matrix $\A'$ obtained from the 
interaction matrix $\A$ by crossing out $j$-th row and $j$-th column for all 
$j\notin {\cal I}$. 
\end{lemma}
Lemma~\ref{subgr} follows from the definition of the process $X(t)$ and the definition of a source subgraph, and is effectively a version of {\em the restriction principle} (see e.g.~\cite{DavisVolk1,DavisVolk2}). Indeed, it suffices to observe that the birth rate for any component $X_v$ depends only the component itself, and
the death rate for a component $X_v$, where $v\in V'$, is determined only by the process's components $X_u$,
$u\in V'$.

\begin{example}
{\rm Consider a LCP $X(t)\in \Z_{+}^8$ the corresponding graph of which is given in Figure~\ref{figure1}. Then a restricted process corresponding to the source subgraph induced by vertices $6,7$ and $8$, i.e.\ $X'(t)=(X_6(t), X_7(t), X_8(t))\in\Z_{+}^3$, is a LCP with transition rates~\eqref{rates} determined by the parameter $\alpha$ and an interaction matrix 
$$
\A'=\begin{pmatrix}0&a_{67}&a_{68} \\ a_{76} &0& a_{78}\\
a_{86} &a_{87}& 0\end{pmatrix},
$$
obtained from an interaction matrix $\A=(a_{ij})=(a_{ij})_{i,j=1}^8$ of the LCP $X(t)$. }
\end{example}

\subsection{The total transition rate}

Recall the total transition rate $R$ defined in~\eqref{R}.
Note that 
\begin{equation}
\label{Rn}
R_n:=R(\zeta(n))={\bf 1}\cdot(\alpha\I+\A)\zeta(n)%=\a |\zeta(n)|+{\bf 1}\cdot\left(\A\zeta(n)\right)
\quad\text{for}\quad \zeta(n)>0,
\end{equation}
where ${\bf 1}=\sum_{i=1}^N\be_i=(1,\dots,1)^\T\in\R^N$ and
 ${\bf I}$ is the $N\times N$ identity matrix.
It is easy to see that  
\begin{equation}
\label{Rn-bound}
R_n\leq \left(\alpha+\max_{j}\theta_j\right)|\zeta(n)|\leq \left(\alpha+\max_{j}\theta_j\right)
(|\zeta(0)|+n)
\quad\text{on}\quad \zeta(n)>0,
\end{equation}
where 
\begin{equation}
\label{theta}
\theta_j=\sum_{i=1}^Na_{ij}\quad\text{for}\quad j=1,\ldots ,N, 
\quad\text{and}\quad |\zeta(n)|=\sum\limits_{i=1}^N\zeta_i(n).
\end{equation}
Although simple upper bounds for $R_n$ in~\eqref{Rn-bound} suffice for our purposes,
we present  in Section~\ref{rate_behaviour}
 some additional  findings concerning the asymptotic behaviour of the total rate $R_n$, which can be of interest in their own right.

%%%%%%%%%%%%%%%%%%%%%%%%%%%%%%%%%%
%%%%%%%%%%%%%%%%%%%%%%%%%%%%%%%%%

\begin{comment}

\begin{definition}
\label{Dpath}
We say that there is a {\em directed path} from $v$ to $w$ on an oriented (sub)graph $G$, and write $v\dirl w$, if there exists a sequence of vertices $v=v_1$, $v_2$, $\dots$, $v_k=w$ of $G$ such that $v_{i-1}\dir v_i$ for $i=1,2,\dots,k-1$. 
\end{definition}

\begin{definition}
\label{Dconnect}
We call a non-empty oriented (sub)graph $G$ {\em connected} if either it consists of just one vertex, or if any two distinct vertices~$v,w\in G$ satisfy $v\dirl w$ and $w\dirl v$ (equivalently, the adjacency matrix of $G$ is irreducible).
\end{definition}

\begin{definition}
\label{Dclosed}
We call a non-empty subgraph $G'$ of an oriented graph $G$ {\em closed}, if there are no $v\in G\setminus G'$ and $v'\in G'$ such that $v\dir v'$ (i.e., there are no directed edges leading into $G'$ from outside).
\end{definition}

\end{comment}

%%%%%%%%%%%%%%%%%%%%%%%%%%%%%%%%%%%
%%%%%%%%%%%%%%%%%%%%%%%%%%%%%%%%%%%%

\subsection{The model semimartingales}

Our next observation is that the dynamics of the LCP $X(t)$ has a striking 
resemblance to that of continuous time multi-type branching process 
$Z(t)=(Z_1(t),\ldots , Z_N(t))\in\Z_{+}^N$, $t\in\R_{+}$ with $N$ types of 
particles, where~$Z_i(t)$ is the number of particles of type~$i$ at time $t$. 
The branching process evolves as follows: after an exponential time with mean 
$1$ a particle of type~$i$ splits to $1+\alpha$ particles of type $i$ and 
$a_{ji}$ particles of type $j$, all split times being independent.

Then, it is easy to see that, given $\bx=(x_1,\ldots ,x_N)\in\Z_{+}^N$, the 
expected change of the $i$-th population of the branching process is 
\begin{equation}
\label{BPdrift}
\E(Z_i(t+\Delta t)-Z_i(t)|Z(t)=\bx)=\left(\alpha x_i+\sum_{j=1}^Na_{ij}x_j\right)\Delta t+\bar o(\Delta t)
\end{equation}
and, similarly, the expected change of the $i$-th component of the LCP 
$X(t)=(X_1(t),\ldots ,X_N(t))\in\Z_{+}^N$,\, $t\in\R_{+}$ with transition 
rates~\eqref{rates} is
\begin{equation}
\label{drift}
\E(X_i(t+\Delta t)-X_i(t)\|X(t)=\bx)=\left(\alpha x_i-\sum_{j=1}^Na_{ij}x_j \right)\Delta t+\bar o\left(\Delta t\right)
\,\text{ on }\, \bx>0,
\end{equation}
where in both cases $\bar o\left(\Delta t\right)\to 0$ as $\Delta t\to 0$. Observe that the right hand side of equation~\eqref{drift} differs from that of equation~\eqref{drift} only by the sign in front of the sum of the interaction terms. Therefore, although the models are different as probabilistic models, they are quite similar to each other algebraically. 
It is well known that scalar products of a multi-type branching process with eigenvectors of the corresponding mean drift matrix play an important role in the study of those processes (\cite{Athreya2}). In light of the similarity between the linear competition process and the multi-type branching processes, it is not surprising that similar  quantities   
are useful in the study of the competition process.

A key observation is the following. Let $\bv$ be a {\em left} eigenvector  
corresponding to an eigenvalue~$\lambda$ of the matrix $\A$, that is 
$$
\bv^\T\A=\l\,\bv^\T.
$$
It follows from equations~\eqref{embedded} that 
\begin{equation}
\label{drift1}
\begin{split}
\E(\zeta_i(n+1)-\zeta_i(n)\|\F_n)&=\frac{\alpha\zeta_i(n)-\sum_{j=1}^Na_{ij}\zeta_j(n)}
{R_n}
%\\
%&
=\frac{\left((\alpha\I-\A)\zeta(n)\right)_{i}}
{R_n}
\quad\text{on}\quad \{\zeta(n)>0\},
\end{split}
\end{equation}
and, hence,
\begin{equation}
\label{eqmart}
\begin{split}
%=\frac{(\alpha-\l_i)\, M_n^{(i)}}{R_n}\quad\text{on}\quad\{\zeta(n)>0\},
\E(\bv\cdot\zeta(n+1)-\bv\cdot\zeta(n)\|\F_n)&=\frac{\bv\cdot (\alpha\I-\A) \zeta(n)}{R_n}%\\
%&
=\frac{(\alpha-\l)\,\bv\cdot\zeta(n)}{R_n}\quad\text{on}\quad\{\zeta(n)>0\}.
\end{split}
\end{equation}
 Thus, the 
 process $\bv\cdot\zeta(n\wedge\sigma)=(\bv\cdot\zeta(n\wedge\sigma),\, n\geq 0)$, where 
\begin{equation}
\label{sigma}
\sigma=\min\left(n\ge 0: \ \min_{i=1,\dots,N}\zeta_i(n)=0\right),
\end{equation}
can be sub- or super-martingale, depending on $\l$ and $\bv$.

In the rest of the section we use  the scalar products $\bv\cdot\zeta(n)$ as building blocks for 
constructing semi-martingales   useful for our proofs and collect some facts concerning the  behaviour 
of these processes.

\bigskip

Let $\lambda_i,\, i=1,\ldots,m$ be distinct eigenvalues 
of the interaction matrix $\A$.
By definition, the interaction matrix $\A$ is non-negative;
also, by our assumptions, $\A$ is irreducible.
Therefore,  by the Perron-Frobenius theorem its largest in absolute value eigenvalue is real, 
strictly positive and simple. Without loss of generality we denote this eigenvalue by $\l_1$.
Let~$\bv_1$
be a left eigenvector of the matrix~$\A$ corresponding to its largest in absolute value eigenvalue $\l_1>0$. 
Therefore, 
 by the same Perron-Frobenius theorem, we can choose the vector 
$\bv_1$ to be strictly positive, that is, $\bv_1\cdot \be_i>0$ for all $i$.

 Let
\begin{equation}
\label{Sn}
V_n:=\bv_1\cdot\zeta(n)=\sum_{i=1}^N (\bv_1\cdot \be_i)\zeta_i(n).
\end{equation}

\begin{remark}
\label{Vn>0}
{\rm 
Note that $V_n\geq 0$, 
 since $\bv_1>0$ and $\zeta_i(n)\geq 0$, $i=1,\ldots ,N$.}
\end{remark}

\begin{proposition}
\label{Vn-super}
Let $V_n$ be the process defined by equation~\eqref{Sn}.
If $\alpha\leq \lambda_1$, then 
$$
\E(V_{n+1}-V_n\|\F_n)\le 0\quad\text{on}\quad\{\sigma>n\},
$$
and if  $\alpha\geq \lambda_1$, then 
$$
\E(V_{n+1}-V_n\|\F_n)\geq 0\quad\text{on}\quad\{\sigma>n\}.
$$
In other words, if $\alpha\leq \lambda_1$, then the process $V_{n\wedge\sigma}$
is a non-negative supermartingale,
 and 
 if $\alpha\geq \lambda_1$, then the process $V_{n\wedge\sigma}$
is a non-negative submartingale.
\end{proposition}
\begin{proof}
The proposition  follows from Remark~\ref{Vn>0} and  equation~\eqref{eqmart}.
%as $\alpha-\l_1\le 0$. 
\end{proof}

Let $\bv$ be a  left eigenvector  corresponding to an eigenvalue $\l\neq \l_1$  of the  matrix of $\A$,  
and let 
\begin{align}
\label{Un}
U_n&=\bv\cdot\zeta(n).
\end{align}

\begin{remark}
\label{Complex}
{\rm  Note that  an eigenvalue $\l\neq \l_1$ of the matrix $\A$ may be 
complex, in which case~$U_n$ is complex as well.}
\end{remark}

\begin{remark}
\label{Un<C}
{\rm 
In what follows, we assume that the Euclidean norm 
 of any eigenvector % vector 
 $\bv$ under consideration 
 is equal to one, i.e.
$||\bv||=1$. Under this assumption we have that 
$\max\limits_{i}|\bv\cdot \be_i|\leq 1$, and, hence, 
$$
|U_n|\leq |\zeta(n)|,
$$
for {\it any} eigenvector $\bv$ of interest. 
This implies that $|U_n|\leq n$,  
as, clearly, $|\zeta(n)|\leq n$.
In addition, note that in the special case of the process $V_n$ 
 defined in~\eqref{Sn} we have, by positivity of the eigenvector $\bv_1$,  the lower bound
$$V_n\geq \left(\min_{i}\bv_1\cdot \be_i\right) |\zeta(n)|,$$
where $\min\limits_{i}\bv_1\cdot \be_i>0$.
}
\end{remark}

\begin{proposition}
\label{Vn>Un}
Let $V_n$ and $U_n$ be the processes defined in~\eqref{Sn} and~\eqref{Un} respectively.
Then 
$\displaystyle{\frac{V_n}{|U_n|}\geq C}$
for some constant $C>0$.
\end{proposition}
\begin{proof}
The proposition follows from Remark~\ref{Un<C}.
\end{proof}

\begin{proposition}
%\label{Un-bound}
\label{Un-subm}
Let $U_n$ be the process defined in~\eqref{Un}.
Then 
$$\E \left[|U_{n+1}| 
\| \F_n\right]\ge \left|1+\frac{\a-\Re(\l)}{R_n}\right|
|U_n|
\quad\text{on}\quad \{\sigma>n\}.$$
In particular, if $\alpha\geq \Re(\lambda)$, then 
\begin{equation}
\label{Un-subm1}
\E \left(|U_{n+1}|-|U_n|\| \F_n\right)\geq 
\frac{\a-\Re(\l)}{R_n}|U_n|
\ge 0\quad\text{on}\quad \{\sigma>n\},
\end{equation}
which implies that the process  $|U_{n\wedge \sigma}|$ is a non-negative 
submartingale.
\end{proposition}
\begin{proof}
By~\eqref{eqmart} we have   that 
\begin{align*}
\E(|U_{n+1}|\|\F_n)&\geq
 \left|\E(U_{n+1}\|\F_n)\right|
= \left|
\left(1+\frac{\a-\l}{R_n}\right)U_n
\right|
\ge 
\left|1+\frac{\a-\Re(\l)}{R_n}\right| |U_n|\quad\text{on}\quad\{\sigma>n\}.
\end{align*}
If $\alpha>\Re(\lambda)$, then 
$\left|1+\frac{\a-\Re(\l)}{R_n}\right|=1+\frac{\a-\Re(\l)}{R_n}$, and 
we get equation~\eqref{Un-subm1}, as claimed.
\end{proof}

To proceed further, denote for short 
$$e:=e(n)=\zeta(n+1)-\zeta(n) \in\{\pm \be_1, \ldots, \pm \be_N\},$$
i.e.  $e$ is the ``increment'' of the configuration. It follows from Remark~\ref{Un<C}, 
that
\begin{equation}
\label{v-e}
|\bv\cdot e|^2\leq 1
\end{equation}
for any eigenvector $\bv$ under consideration.
In addition, we will use the elementary 
inequality\footnote{indeed, we have
$1+x-y+x^2+3y^2-{\frac {1+x}{1+y}}
= \big( x+{\frac {y}{2+2
y}}\big)^2+{\frac{y^2\left(6y+7 \right)\left(2y+1
 \right) }{4\left( 1+y \right)^2}}$, 
 which readily implies~\eqref{1+x/1+y}}
\begin{equation}
\label{1+x/1+y}
 \frac{1+x}{1+y} \leq 1+x-y+x^2+3y^2
  \quad \text{ for }
  % \max\{|x|,|y|\}\leq \frac{1}{3}.
  |y|\leq \frac{1}{2}
  \text{ and all }x\in \R.
\end{equation}

\begin{proposition}
\label{Wn-supermartingale}
Let 
$V_n$ be the process defined in~\eqref{Sn}, and let 
 $U_n$ be the process defined in~\eqref{Un},
where the eigenvalue $\l\neq \l_1$.
Define 
\begin{equation}
\label{Wn}
W_n=\frac{V_n}{|U_n|^2}.
\end{equation}
If $\alpha>\l_1$ (and, hence, $\alpha>\Re(\l)$), 
then there exist constants $C_1>0$ and $\gamma>0$
such that 
$$
\E(W_{n+1} -W_n \| \F_n)\le  
0\quad\text{on}\quad \{\sigma>n, \, W_n \le C_1,\, |\zeta(n)|>\gamma\},
$$ 
i.e., the process $W_{n\wedge \sigma}$ is a non-negative supermartingale on 
$\{W_n \le C_1, |\zeta(n)|>\gamma\}$.
\end{proposition}

\begin{proof}
Observe that 
\begin{equation}
\label{Wn+1}
\begin{split}
W_{n+1}&=\frac{V_n+\bv_1\cdot  e}{|U_n +\bv \cdot e|^2}
=\frac{V_n+\bv_1\cdot e}{|U_n |^2+2\Re(U_n 
\overline{\bv}\cdot e)+1}\\
&=W_n \left(1+\frac{\bv_1\cdot e}{V_n}\right)
\left(1+\frac{2\Re(U_n (\overline{\bv}\cdot e))+1}{|U_n|^2}\right)^{-1}\\
\end{split}
\end{equation}
We can assume that both $V_n$ and $|U_n|$ are sufficiently large.
Indeed, by Remark~\ref{Un<C} we have that 
$V_n$ is large provided that $|\zeta(n)|$ is large, since they are 
of the same order. 
Then, the  assumption $\{W_n\le C_1\}=
\{V_n\le C_1|U_n|^2\}$ implies that $|U_n|$ is also large.
Therefore, using~\eqref{v-e} and~\eqref{1+x/1+y} we obtain that 
\begin{equation}
\label{Wn+1-2}
W_{n+1} \leq W_n(1+J_{n,1}+J_{n,2}+J_{n,3}+J_{n,4}),
\end{equation}
where 
\begin{align*}
J_{n,1}&=\frac{\bv_1\cdot e}{V_n},\quad
J_{n,2}=-\frac{2\Re(U_n (\overline{\bv}\cdot e))}{|U_n |^2},\\
J_{n,3}&=\frac{1}{V_n^2}\quad\text{and}\quad
J_{n,4}= 3\frac{(2\Re(U_n)+1)^2}{|U_n |^4}.
\end{align*}
Thus, 
\begin{equation}
\label{Exp}
% \begin{split}
% \E(W_{n+1}-W_n|\F_n) &\leq W_n\E(J_{n,1}|\F_n)+W_n\E(J_{n,2}|\F_n)\\
% &+W_n\E(J_{n,3}|\F_n)+W_n\E(J_{n,4}|\F_n).
% \end{split}
\E(W_{n+1}-W_n|\F_n) \leq W_n\E(J_{n,1}|\F_n)+W_n\E(J_{n,2}|\F_n)
+W_n\E(J_{n,3}|\F_n)+W_n\E(J_{n,4}|\F_n).
\end{equation}
Using equation~\eqref{eqmart}, we obtain  that 
\begin{align}
\label{j1}
W_n\E(J_{n,1}|\F_n)&=\frac{W_n}{V_n} \frac{(\alpha-\lambda_1)V_n}{R_n}=
(\alpha-\lambda_1)\frac{W_n}{R_n},\\
W_n\E(J_{n,2}|\F_n)&=-2\frac{W_n U_n }{|U_n |^2}
\frac{(\alpha-\Re(\lambda))\overline{U}_{n}}{R_n}=
-2(\alpha-\Re(\lambda))\frac{W_n}{R_n}.
\label{j2}
\end{align}
Further,  since both $V_n$ and $|U_n|$ 
are sufficiently large, 
we have that    
$$W_n\E(J_{n,3}|\F_n)=\frac{1}{V_n^2}\frac{W_n}{R_n}\leq \eps_1\frac{W_n}{R_n},$$
and 
$$W_n\E(J_{n,4}|\F_n)=3
\frac{(2\Re(U_n)+1)^2}{|U_n |^4}\frac{W_n}{R_n}\leq \eps_2\frac{W_n}{R_n},$$
% O(|U_n|^{-2})\frac{W_n}{R_n},$$
for  a sufficiently small $\eps_1>0$.
Therefore, since $\alpha>\l_1>\Re(\l)$, we get that 
$$
\E(W_{n+1} -W_n \| \F_n)\le (-\alpha-\lambda_1+2\Re(\lambda)+2\eps_1)\frac{W_n}{R_n} 
\le 0,$$
as claimed.
\end{proof}

\begin{proposition}
\label{lemSPA3}
Let $V_n$  and $U_n$ be the processes from Proposition~\ref{Wn-supermartingale}
and define
\begin{equation}
\label{Hn}
H_n=\frac{V_n}{|U_n|}.
% =W_n|U_n|.
\end{equation}
There exist constants $C_3,C_4>0$ and $\gamma>0$ such that 
$$
\E(H_{n+1}-H_n\| \F_n)\le -\frac{C_4}{R_n}\leq 0\quad\text{on}\quad
\left\{ \sigma>n, \,|V_n|\leq C_3|U_n|^2,\, |\zeta(n)|>\gamma\right\}.
$$
% i.e., the process $H_{n\wedge \sigma}$ is a non-negative 
% supermartingale on 
% $\left\{ \sigma>n, \,|V_n|\leq C_3|U_n|^2,\,|\zeta(n)|>\gamma\right\}$.
\end{proposition}
\begin{proof}
Write (by~\eqref{1+x/1+y})
\begin{align*}
H_{n+1} &=
 \frac{V_n+\bv_1\cdot e}{|U_n|+|U_{n+1}|-|U_n|}= H_n
 \frac{1+\frac{\bv_1\cdot e}{V_n}}
 {1+ \frac{|U_{n+1}|-|U_n|}{|U_n|}}.
\end{align*}
Using~\eqref{v-e} and~\eqref{1+x/1+y}, we obtain that
\begin{align*}
H_{n+1} &\leq H_n\Big(1 + \frac{\bv_1\cdot e}{V_n}
 - \frac{|U_{n+1}|-|U_n|}{|U_n|}
 +\frac{(\bv_1\cdot e)^2}{V_n^2}
 +3\frac{(|U_{n+1}|-|U_n|)^2}{|U_n|^2} \Big)\\
  & \leq H_n\Big(1 + \frac{\bv_1\cdot e}{V_n}
 - \frac{|U_{n+1}|-|U_n|}{|U_n|}
 +\frac{1}{V_n^2}
 +\frac{3}{|U_n|^2} \Big) .
\end{align*}
So, we now find, with~\eqref{eqmart} and Corollary~\ref{Un-subm},  that 
\begin{align*}
  \E(H_{n+1}-H_n\mid \F_n)
& \leq H_n\left(\frac{\alpha-\lambda_1}{R_n} - 
\frac{\alpha-\Re(\lambda)}{R_n}
+\frac{1}{V_n^2} + \frac{3}{|U_n|^2}\right)\\
& = \frac{H_n}{R_n}\left(-(\lambda_1-\Re(\lambda))
+\frac{R_n}{V^2_n} + \frac{3R_n}{|U_n|^2}\right).
\end{align*}
Next, using bound~\eqref{Rn-bound}, Remark~\ref{Vn>0}  and the assumption that 
$V_n\leq C_3|U_n|^2$,  we obtain that
$$
\frac{R_n}{V^2_n}\leq \eps_1\quad\text{and}\quad
\frac{3R_n}{|U_n|^2}\leq \frac{3C_3R_n}{V_n}\leq \eps_1,
$$
where $\eps_1>0$ is sufficiently small (provided that $|\zeta(n)|$ is sufficiently large and $C_3$ 
is sufficiently small), so that 
$-(\lambda_1-\Re(\lambda))+2\eps_1<0$.
By Proposition~\ref{Vn>Un} we have that 
$H_n$ is bounded below by a positive constant. Therefore, we finally obtain that 
$$
\E(H_{n+1}-H_n\mid \F_n)\leq -\frac{C_4}{R_n},$$
for some constant $C_4>0$,  as  claimed.
\end{proof}

\section{Proofs of theorems}\label{sec-proofs}
\subsection{Proof of Theorem~\ref{T1}}
The proof of Theorem~\ref{T1} is based on Lemmas~\ref{lemgraphtheory} and~\ref{subgr} (see Section~\ref{graph}) and Lemmas~\ref{lem_N2} and~\ref{sigma_finite} stated below and proved later in 
Section~\ref{sec-lemmas_proofs}.

\begin{lemma}\label{lem_N2}
Let $N=2$ and $\A=\begin{pmatrix}0 &0 \\ \b &0\end{pmatrix}$, where $\beta>0$ is a given constant. In other words, the model graph $G$ contains just one edge $1\dir 2$. Suppose that $\zeta_1(0)>0$ (equivalently, $X_1(0)>0$). Then, with probability one, the DTMC (equivalently, CTMC) dies out on vertex $2$, i.e.\ there is a (random) time $n'\ge 0$ such that $\zeta_2(n)=0$ for all $n\ge n'$.
\end{lemma}

\begin{lemma}
\label{sigma_finite}
Let the interaction matrix $\A$ be irreducible, or, equivalently, the corresponding directed graph $G$ (defined in Definition~\ref{D2}) be a strongly connected directed graph. Recall the stopping time $\sigma$ defined in~\eqref{sigma} and define similarly 
\begin{equation*}
\ts=\min\left(t\ge 0: \ \min_{i=1,\dots,N}X_i(t)=0\right).
\end{equation*}
Then for any initial condition $X(0)=\zeta(0)$, 
it holds that
$\P(\sigma<\infty)=\P(\ts<\infty)=1.$ 
\end{lemma}

The proof of the theorem is by induction on $N$. If $N=1$, then the statement is trivial. Assume that $N\ge 2$. By Lemma~\ref{lemgraphtheory} there is a strongly connected source subgraph $G'=(V', E')$ with $N'$ vertices. Now there are two possibilities.

\begin{itemize}
\item[(a)] If $N'\ge 2$, then we can apply Lemma~\ref{sigma_finite}. After one of the components $X_v,\, v\in V'$, becomes~$0$, say, it is $X_{v'}$, we remove the vertex $v'$ from $V$, and, hence, remove the corresponding column and row of the matrix $\A$. New graph contains $N-1$ vertex for which the statement of the theorem holds by induction. \item[(b)] If $N'=1$, then $G'$ consists of just one source vertex, say, $v$. Since, by definition, there are no edges incoming to $v$, the death rate at $v$ is zero. Therefore, the component $X_v(t)$ will survive forever.

Further, there are two sub-cases to consider. First, if $v$ is an isolated vertex of $G$ (i.e. there are no edges coming into or going out of $v$), then we can apply the induction to the subgraph induced by the vertex set $V\setminus\{v\}$.

If the vertex $v$ is not isolated, then consider a vertex $w$ for which $v\dir w$. Since the birth rate at $v$ depends on $X_v$ only, and the death rate at $w$ results from the weighted sum of $X_{u}$ for all $u$ such that $u\dir w$, one can couple $(X_v(t),X_w(t))$ with CTMC $(\tilde X_v(t),\tilde X_w(t))$ on the graph $\tilde G$ with two vertices $\{v,w\}$ and the only edge $v\dir w$ in such a way that 
$$
X_v(t)=\tilde X_v(t),\quad X_w(t)\le \tilde X_w(t),
$$ 
similarly to Lemma~\ref{subgr}. The above inequality arises from the fact that there may be some vertex $u$ such that $u\dir w$. By Lemma~\ref{lem_N2} the LCP on $\tilde G$ dies out on $w$, and, hence, the same happens on $G$. Therefore, we can remove the vertex $v$ and all vertices $w$ such that $v\dir w$ from the graph. The resulting graph will have $\le N-2$ vertices, so we can apply the induction again by introducing a sequence of stopping times, each time identifying a random smaller subgraph, and using  the fact that we are allowed arbitrary initial conditions in Lemmas~\ref{lem_N2} and~\ref{sigma_finite}.
\end{itemize}

\begin{remark}
{\rm It is crucial that one chooses the strongly connected {\bf source} subgraph. Indeed, consider the graph in Figure~\eqref{figure1}, and assume that $X_i(0)=0$ for $i=5,6,7,8$. One might think that the subgraph with vertices $\{2,3,4\}$ is admissible. Indeed, at a first glance, it looks like the chances that $X_2$ dies out only ``improve'' due to the presence of the link $1\dir 2$. However, this becomes not so apparent, if one considers the fact that the lower value at $\{2\}$ results in higher values at $\{3\}$, which, in turn, leads to a lower value at $\{4\}$, and as a result, a smaller death rate at $\{2\}$.}
\end{remark}

\subsection{Proof of Theorem~\ref{T2}}\label{ProofT2}
As before, we assume that all $a_{ij}$, $\a$ and $Y_i(0)$'s are integers. This implies that all $Y_i(n)$'s are integers for all $n\geq 1$. The proof can be easily adapted to the case when it is not true, but we do not want to complicate it unnecessary.

We start with explaining the connection between the urn model and the linear competition process from Theorem~\ref{T1}. Note that, as long as the process is {\it sufficiently far away} from the boundary, (i.e. all $Y_i'$s are sufficiently large) we have that 
$$
\E\big(Y_i(n+1)-Y_i(n)\|Y(n)=(Y_1,\ldots , Y_N)\big)
=\frac{\alpha Y_i-\sum_{j=1}^N 
a_{ij} Y_j}{\sum_{j=1}^N Y_j},\quad i=1,\ldots ,N,
$$
where the numerator looks the same as the numerator on the right hand side of equation~\eqref{drift1}, giving the mean jump of a component of the DTMC $\zeta(n)$ when $\zeta(n)>0$. Therefore, the proof can be carried out similarly to that of Theorem~\ref{T1}, with a little bit more work. A new argument required is given below.

Assume that either $N=2$ and matrix $\A=\begin{pmatrix}0 &0 \\ \b &0\end{pmatrix}$, where $\beta>0$ (i.e.
as in Lemma~\ref{lem_N2}), or the matrix $\A$ is irreducible. Let
\begin{equation}
\label{lambda}
B=\max_{i,j} a_{ij},
\end{equation}
and define the following sets 
\begin{align*}
D&=\{\bx\in \Z_+^N:\ \text{at least one }x_i<B\},\\
D_0&=\{\bx\in \Z_+^N:\ \text{at least one }x_i=0\}.
\end{align*}
Similarly to Lemma~\ref{lem_N2} and Lemma~\ref{sigma_finite}, one can show that, with probability one, the DTMC $Y(n)$ enters set $D$. Hence, if $Y(n)$ leaves $D$ without hitting $D_0$, it will have to re-enter $D$ again. Let us show that it is impossible to enter and leave $D$ infinitely many times without hitting $D_0$, i.e.\ 
$$
\{Y(n)\in D\text{ infinitely often}\}\subseteq \{Y(m)\in D_0\text{ for some }m\}.
$$
Assume that $a_{12}>0$, and, hence $a_{12}\ge 1$ ($a_{ij}$ are integers by the earlier assumption). Note that if $a_{1k}$ were $0$ for all $k$ then $Y_1(n)$ would never decrease. Now, at time $m>n$, as long as $Y(m)\in D\setminus D_0$, the conditional probability that ball of type $2$ is chosen given a ball of either type $1$ or $2$ is chosen, is at least $\frac{1}{1+(B -1)}=\frac{1}{B}$, provided $Y_2(m)\ge 1$ (otherwise $Y_2(m)=0$, i.e.\ the process has already reached the boundary $D_0$). On the other hand, on this event 
$$
Y_1(m+1)=\max(Y_1(m)-a_{12},0)\le Y_1(m)-1.
$$
Hence, if this conditional event happens consecutively (at most) $B-1$ times, $Y_1$ will become zero. Since all types of balls (as long as they are present in the urn) are chosen infinitely often (the process does not explode), we obtain that
$$
\P(\exists m\ge n:\ Y(m)\in D_0\text{ and } Y(k)\in D\text{ for all }k\in[n,m]\| Y(n)\in D)\ge 
B^{-(B -1)}>0.
$$
The claim now follows.

The rest of the proof of Theorem~\ref{T2} is identical to that of Theorem~\ref{T1}.

\section{Proofs of lemmas}\label{sec-lemmas_proofs}
\subsection{Proof of Lemma~\ref{lem_N2}}
We start with describing the intuition behind the proof. The behaviour of the LCP should be similar to that of the dynamical system, governed by the system of differential equations
$$
\begin{cases}
\dot{x}=x,\\
\dot{y}=y-\beta x.
\end{cases}
$$
The solution to this system is $x=C_1 e^t$, $y=(C_2-\beta C_1 t)e^t$. It is clear that there are no constants~$C_1$ and~$C_2$ for which both $x(t)$ and $y(t)$ would remain positive for all $t>0$.

The formal proof is as follows. First, assume w.l.o.g.\ that $\alpha=1$, and note that the DTMC $\zeta(n)=(\zeta_1(n),\zeta_2(n))$ can be coupled with the classical P\'olya urn with two types of balls such that
$$
 \zeta_1(n)\ge \tilde\zeta_1(n),\quad \zeta_2(n)\le \tilde\zeta_2(n),
$$
where $\tilde\zeta_1(n)$ and $\tilde\zeta_2(n)$ denote the number of balls of type 1 and 2 respectively (if $\beta=0$ then our model will be exactly the P\'olya urn). The well-known result (see e.g.~\cite{Blackwell}) says that, with probability one, $\tilde\zeta_2(n)/\tilde\zeta_1(n)\to \xi/(1-\xi)$, where $\xi$ is a Beta-distributed random variable, hence
\begin{align} \label{eqq-0}
\limsup_{n\to\infty} \frac{\zeta_2(n)}{\zeta_1(n)}<\infty\quad\text{a.s.}
\end{align}

Now we will use the following martingale argument. Let $S_n=\frac{\zeta_2(n)}{\zeta_1(n)+\zeta_2(n)}$.  Then 
\begin{align}\label{eqq-1}\nonumber
\E(S_{n+1}-S_n\| \zeta(n)=(x,y))&=\frac{
x \left[\frac{y}{x+y+1}-\frac{y}{x+y}\right]
+
y\left[\frac{y+1}{x+y+1}-\frac{y}{x+y}\right]
+
\b x\left[\frac{y-1}{x+y-1}-\frac{y}{x+y}\right]
}{(1+\b)x+y}
\\
&=\frac{-\b x^2}{(x+y)(x+y-1)(x+y+\b x)}\quad \quad\text{when }x,y\ge 1,
\end{align}
and
\begin{align}\label{eqq-1b}
\E(S_{n+1}-S_n\| \zeta(n)=(x,0))&=0 \quad\text{when }x\ge 1.
\end{align}
Suppose that $1\le y\le rx$ for some positive $r>0$. Then the conditional expectation in~\eqref{eqq-1} is bounded  above by
\begin{align}\label{eqq-2}
\frac{-\b x^2}{(x+rx)(x+rx-1)(x+rx+\b x)}
\leq
\frac{-\b x^2}{(x+rx)^2(x+rx+\b x)}
= -\frac {C_r}{x}\ ,
\end{align}
where $C_r=\frac{\b}{(1+r)^2(1+r+\b)}>0$. 

Let 
\begin{align*}
\tau_*&=\inf\{n\ge 0: \ \zeta_2(n)=0\},\\
\tau_r&=\inf\{n\ge 0: \ \zeta_2(n)>r\zeta_1(n)\}.
\end{align*}
From~\eqref{eqq-1}, \eqref{eqq-1b}, \eqref{eqq-2}, and the fact that $\zeta_1(n)\le n+\zeta_1(0)$, we obtain
\begin{align}\label{eqq-3}
\E(S_{n+1}-S_n\| \F_n)&\le 
-\frac{C_r\, 1_{\{\tau_*,\tau_r>n\}}}{\zeta_1(n)}
\le  -\frac{C_r \, 1_{\{\tau_*,\tau_r>n}\}}{\zeta_1(0)+n}
\end{align}
By taking the expectation and summing up \eqref{eqq-3} for all $n\ge 0$, and noting that $\E(S_{n+1}-S_n\| \F_n)\le 0$, so that $S_n$ is a non-negative supermartingale,  we obtain
$$
0\le \lim_{n\to\infty} \E(S_n)\le \E(S_0)-\sum_{n=0}^{\infty} \frac {C_r\P(\min(\tau_*,\tau_r)>n)}{n+\zeta_1(0)}
$$
for all $r>0$. Since the harmonic series diverges, we have $\P(\min(\tau_*,\tau_r)>n)\downarrow 0$, that is, $\min(\tau_*,\tau_r)<\infty$ a.s. On the other hand, it follows from~\eqref{eqq-0}  that a.s.\ there is an $r>0$ such that $\tau_r=\infty$. 
Consequently, $\tau_*<\infty$ a.s., that is, eventually $\zeta_2(n)$ will become $0$.
\begin{remark}
{\rm Note that the dynamics described in Lemma~\ref{lem_N2} is similar to that of a triangular urn, studied in~\cite[Theorem 2.3]{PV99}.}
\end{remark}

\subsection{Proof of Lemma~\ref{sigma_finite}}
We are going to show that, with probability one, the stopping time $\sigma$ (defined in~\eqref{sigma})
 is finite. 
This will imply that $\ts$ is also finite almost surely, since the LCP $X(t)$ is a non-explosive CTMC. 
Below, we  consider two cases: $\alpha\leq \l_1$ and $\alpha>\l_1$.

\subsubsection{Case: $\alpha\leq \l_1$}
Recall the process $V_n$ defined in~\eqref{Sn}.
By Proposition~\ref{Vn-super}, if $\alpha\leq \l_1$, then the process $V_{n\wedge\sigma}$ 
is a non-negative supermartingale, and, hence, it must converge a.s. Note that 
if $\sigma>n$, then at least one of the following events 
$\{\zeta_i(n+1)-\zeta_i(n)=\pm 1\}$, $i=1,\ldots ,N$, must occur, and, hence, 
$V_n=\bv_1\cdot\zeta(n)$ will change at least by $\eps=\min_{j=1,\dots,N} 
\bv_1\cdot\be_j>0$. Therefore, convergence of $V_{n\wedge\sigma}$ is possible if 
and only if the stopping time $\sigma$ is finite.

\begin{remark}
\label{linear-time}
{\rm In addition, note that if $\alpha<\l_1$ then, by equation~\eqref{eqmart}
 we have that 
\begin{equation}
\label{rho_drift}
\E(V_{n+1}-V_n\|\F_n)=-(\l_1-\alpha) \frac{\bv_1\cdot \zeta(n)}{{\bf 1}\cdot(\alpha\I+\A)\zeta(n)}\le 
-(\l_1-\alpha)\rho<0\quad\text{on}\quad\{\sigma>n\},
\end{equation}
where
$$
\rho=\inf_{\bx>0}
\frac{\bv_1\cdot\bx}{{\bf 1}\cdot(\alpha\I+\A)\bx}=\min_{\bx>0}\frac{\bv_1\cdot\bx}{{\bf 1}\cdot(\alpha\I+\A)\bx}
=\min_{j=1,\dots,N} \frac{\bv_1 \cdot \be_j}{\a+\theta_j}>0,
%{\a+\theta_j\sum_{k=1}^N a_{kj}}>0
$$
and $\theta_j$ is defined in~\eqref{theta}, 
since the right-most numerator is bounded below by $\eps>0$. 
In turn, equation~\eqref{rho_drift} implies
(by Theorem 2.6.2 in \cite{MPW}) that 
\begin{equation}
\E(\sigma)\leq\frac{V_0}{(\l_1-\alpha)\rho}=\frac{\bv_1\cdot\zeta(0)}{(\l_1-\alpha)\rho}\quad\text{for}\quad \zeta(0)>0.
\end{equation} 
In other words, if $\l_1>\alpha$, then the waiting time until extinction is linear in the initial position of the process. }
\end{remark}

\subsubsection{Case: $\alpha>\l_1$}
We start with briefly outlining the plan of the proof in this case.
Given $\psi>0$ define 
\[
D_{\psi}=
\Big\{\zeta=(\zeta_1,\ldots,\zeta_N)>0: \min_{j=1,2,\dots,N}\frac{ \zeta_j}{|\zeta|}\le \psi\Big\}
\subset\Z_{+}^N.
\]
For suitable constants $\gamma,\psi$, we will define 
the moment~$\sigma'(m)$ when the process steps on the ``bad set''
$D_\psi\cup \{\zeta: |\zeta|\leq \gamma\}$ after time~$m$:
\begin{equation}
\label{df_sigma'}
 \sigma'(m) = \min\big\{n\geq m: \zeta(n) \in D_\psi\text{ or } |\zeta|\leq \gamma\big\}
\end{equation}
(see Figure~\ref{f_bad_set}); note also that $\sigma'(0)\leq \sigma$.
\begin{figure}
\centering
 \includegraphics{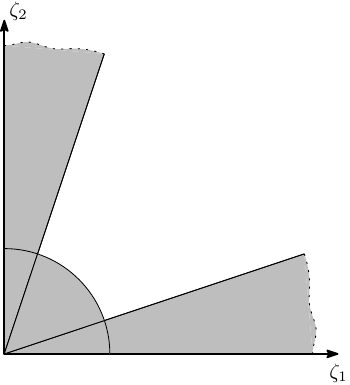}
 \caption{\small{An illustration of the ``bad set'' (colored gray)
 in the case $N=2$.}}
\label{f_bad_set}
\end{figure}
We will then show that if we are able to prove that~$\sigma'(m)$ 
is a.s.\ finite for any ``initial'' (at time~$m$) configuration of the 
process, then this would imply that~$\sigma$ is a.s.\ finite.
The advantage of working with $\sigma'$ instead of~$\sigma$
is that, as long as the process is not in the ``bad set'',
the configuration is ``balanced'' (i.e., all its components are of the same order)
and ``not too small'', which makes several Lyapunov-finctions-related
computations (needed to prove the finiteness of~$\sigma'$) more smooth.

\paragraph{Step 1: $\sigma'(\cdot)<\infty$ implies  $\sigma<\infty$.}
%Recall that $|\zeta|=\zeta_1+\cdots+\zeta_N$ 
\begin{lemma}
\label{l_kill_the_weak}
There are positive constants $\psi, \gamma_2, \gamma_3$ (depending on the model's parameters) such that
\begin{equation}
\label{eq_kill_the_weak}
\P_\zeta(\sigma>\gamma_2 |\zeta|)\leq e^{-\gamma_3 |\zeta|}
\quad\text{for}\quad \zeta\in D_{\psi},
\end{equation}
where $\P_\zeta$ denotes  the distribution of the discrete-time process started at 
$\zeta=(\zeta_1,\ldots,\zeta_N)$.
\end{lemma}
In other words,  if the process finds itself in a ``very unbalanced'' state (i.e., the relative difference between some of its coordinates is very large), then the moment~$\sigma$ should happen \emph{relatively} soon, in time of order of the size of that state.

\begin{proof}
Let us observe that, when the process is currently in a configuration $\zeta>0$, the (discrete-time) dynamics described in~\eqref{embedded} can be reformulated in the following way:
\begin{itemize}
 \item[(i)] Choose a coordinate $j\in\{1,\ldots,N\}$ 
 with probability 
 $$\frac{(\alpha+\theta_j)\zeta_j}{\sum_{k=1}^N(\alpha+ \theta_k)\zeta_k}
 =\frac{(\alpha+\theta_j)\zeta_j}{R(\zeta)},$$
 where $\theta_j$ is defined in~\eqref{theta};
 we will say that $j$ was ``chosen to act''.
 \item[(ii)] The chosen coordinate will then ``increase itself''
 by one unit (i.e., $\zeta$ goes to $\zeta+ \be_j$)
 with probability $\frac{\alpha}{\alpha+\theta_j}$,
 and it will ``do a kill at~$i$'' (i.e., $\zeta$ goes to $\zeta-\be_i$)
 with probability $\frac{a_{ij}}{\alpha+\theta_j}$.
\end{itemize}
Let $\mu\geq 1$ be such that $\mu^{-1}\leq \alpha+\theta_j\leq\mu$
for all $j=1,\ldots,N$.
Also, let 
\[
\delta=
\min_{i,j:a_{ij}>0} \frac{a_{ij}}{\alpha+\theta_{j}} 
\in (0,1)
\]
be the minimum ``individual conditional killing probability'' (recall the second part in~(ii) above) among those that are positive. Let us also say that an event holds with very 
high probability (w.v.h.p.) if its $\P_\zeta$-probability is at least  $1-e^{-\gamma' |\zeta|}$ for some constant~$\gamma'>0$ (so that we need to prove that the moment~$\sigma$ will happen in linear in~$|\zeta|$ time w.v.h.p.). Let us first prove the following fact: 
\begin{equation}
\label{kill_edge}
%  \text{if } a_{ji}>0,\ \frac{\zeta_i}{|\zeta|}\geq h\in (0,1),\  
%   \frac{\zeta_j}{|\zeta|}\leq \frac{\delta^2 h^2}{16}, \text{ then }
%   \sigma\leq \frac{\delta h}{3}|\zeta| \; \text{ w.v.h.p.}
 \text{if } a_{ji}>0,\ \frac{\zeta_i}{|\zeta|}\geq h\in (0,1),\  
  \frac{\zeta_j}{|\zeta|}\leq \frac{\delta^2 h^2}{16\mu^{6}}, 
  \text{ then }
  \sigma\leq \frac{\delta h}{3\mu^{4}}|\zeta| \; \text{ w.v.h.p.}
\end{equation}
To see the above, let us think about what will happen at~$i$ and~$j$ during the 
time~$t|\zeta|$ (later we will suitably choose~$t$ to be equal 
to~$\frac{\delta h}{3\mu^{4}}$,
but assume for now that $t<h$). 
First, note that
\[
 \mu^{-2} \frac{\zeta_k}{|\zeta|}
 \leq \P_\zeta(\text{$k$th coordinate is chosen to act})
\leq \mu^{2} \frac{\zeta_k}{|\zeta|}.
\]
Then, during the first~$t|\zeta|$ time units, 
the proportion~$\frac{\zeta_i(\cdot)}{|\zeta(\cdot)|}$ at the $i$th coordinate cannot become less than $h-t$ (think of the ``worst case'' when the $i$th coordinate always decreases by~$1$ per time unit), and therefore the probability that the $i$th coordinate is chosen to act cannot become less than $\mu^{-2}(h-t)$, so the (conditional) probability that it causes the $j$th coordinate to decrease on a given step (among the first~$t|\zeta|$) is at least $\delta\mu^{-2}(h-t)$. Therefore, w.v.h.p., the ``killing count'' (during the first~$t|\zeta|$ time units) of~$i$ at~$j$ will be at least $(\delta\mu^{-2}(h-t)-\eps)t|\zeta|$ (where~$\eps$ is fixed-and-small). 

On the other hand, the proportion  at the $j$th coordinate cannot exceed $\frac{\zeta_j}{|\zeta|}+t$(analogously, think of the ``best case'' when the $j$th coordinate always increases by~$1$ per time unit), and therefore the probability
that the $j$th coordinate is chosen cannot become larger than $\mu^{2}(\frac{\zeta_j}{|\zeta|}+t)$; so, w.v.h.p.\ the increment
that the $j$th component causes at~$j$ will not  exceed~$(\mu^{2}(\frac{\zeta_j}{|\zeta|}+t)+\eps)t|\zeta|$. Now (note also that the $j$th coordinate starts from $\zeta_j=\frac{\zeta_j}{|\zeta|}\times|\zeta|$)
we observe that
\begin{equation}
\label{eps-h}
 (\delta\mu^{-2}(h-t)-\eps)t 
> \Big(\mu^{2}\Big(\frac{\zeta_j}{|\zeta|}+t\Big)+\eps\Big)t
+ \frac{\zeta_j}{|\zeta|}
\end{equation}
for $t=\frac{\delta h}{3\mu^4 }$
and with small enough $\eps>0$ provided that
$\frac{\zeta_j}{|\zeta|}\leq \frac{\delta^2 h^2}{16\mu^6}$,
which implies~\eqref{kill_edge} with the usual large-deviation 
estimates.\footnote{Note that,
with $\eps=0$, the left-hand side of~\eqref{eps-h}  is at least 
$\frac{2\delta^2 h^2}{9\mu^6}$
while the right-hand side is 
$\frac{\zeta_j}{|\zeta|}(1+\frac{\delta h}{3\mu^2})+\frac{\delta^2 h^2}{9\mu^6}
\leq\frac{4\zeta_j}{3|\zeta|}+\frac{\delta^2 h^2}{9}$.
}
Now, set 
\[
 \psi = \frac{\delta^{2^N-2}}{16^{2^N-1}\mu^{3(2^N-2)}N^{2^N-1}}
  \text{ and } \gamma_2 = \frac{\delta }{3\mu^4 N}.
\]
Then, \eqref{kill_edge} implies~\eqref{eq_kill_the_weak}
in the following way: assume that $\zeta_{i_0}=\max_k \zeta_k$
and $\zeta_{j_0}=\min_k \zeta_k$. 
Consider also the recursion~$b_0=N^{-1}$,
$b_{k+1}=\frac{\delta^2}{16\mu^{6}} b_k^2$ and 
observe that $b_{N-1}=\psi$. Then, $\frac{\zeta_{i_0}}{|\zeta|}\geq \frac{1}{N}$
and there exists a ``path'' (along the ``oriented edges'' $(i,j)$ such that
$a_{ji}>0$) of length at most~$N-1$ from~$i_0$ to~$j_0$. 

W.l.o.g.\ assume $i_0=1$, and that the path is vertices $1, 2, \ldots, j_0$. Find the smallest $k$ for which $\zeta_{k+1}/|\zeta| \leq b_{k+1}$ (note that if it does not exist then $\zeta_{j_0}/|\zeta|>\psi$). Consider the pair of  vertices $(k,k+1)$. We know that
$$
\frac{\zeta_{k}}{|\zeta|}\ge  b_k=:h
\quad\text{and}\quad
\frac{\zeta_{k+1}}{|\zeta|}\leq b_{k+1}=\frac{\delta 
b_k^2}{16\mu^{6}}
=\frac{\delta h^2}{16\mu^{6}}.
$$
Now we can apply~\eqref{kill_edge} to conclude the proof of Lemma~\ref{l_kill_the_weak}.
\end{proof}

%%%%%%%%%%%%%%%%%%%%%%%%%%%%%%%%
%%%%%%%%%%%%%%%%%%%%%%%%%%%%%%%%%%%

%%%%%%%%%%%%%%%%%%%%%%%%%%%%%%%%
%%%%%%%%%%%%%%%%%%%%%%%%%%%%%%%

Next, we formulate the fact announced in the beginning
of this subsection: 
\begin{corollary}
\label{c_sigma_sigma'}
Assume that $\psi$ is the constant from Lemma~\ref{l_kill_the_weak},
and let~$\sigma'(\cdot)$ be defined as in~\eqref{df_sigma'}.
Then, $\sigma'(m_0)<\infty$ a.s.\ for any $\zeta(m_0)$
implies that $\sigma<\infty$ a.s.
\end{corollary}
\begin{proof}
Notice that, due to Lemma~\ref{l_kill_the_weak},
the fact that the set $\{\zeta: |\zeta|\leq \gamma\}$ 
is finite for any $\gamma$, and the Strong Markov
Property, there is $\eps_0>0$ such that
on $\big\{\zeta(m_0)\in D_\psi\cup \{\zeta: |\zeta|\leq \gamma\}\big\}$
there is a $\F_{m_0}$-measurable random variable~$\Xi$ such that
$\sigma\leq m_0+\Xi$ with probability at least~$\eps_0$.
The claim now follows from the usual
argument of the sort ``if you keep tossing 
a coin, it will eventually come heads''.
\end{proof}

\paragraph{Step 2: $\sigma'(\cdot)<\infty$ almost surely.} 

Recall that $\lambda_1>0$ is the largest in absolute value eigenvalue of the interaction 
matrix $\A$. For the  rest of the proof we fix an eigenvalue $\lambda\neq \lambda_1$ 
(which can be complex), 
a left eigenvector~$\bv$ corresponding to the eigenvalue~$\l$ and 
consider the process 
\begin{equation}
\label{Un-new}
U_n=\bv\cdot\zeta(n),
\end{equation}
i.e., as in Proposition~\ref{Un-subm}.

We start with proving that, roughly speaking,
the process~$|U_n|$ will eventually
outgrow~$\sqrt{n}$:
\begin{proposition}
\label{cla}
For any constant $A>0$, any initial configuration, and any~$m_0$
it holds that
$$
\P\big(\sigma'(m_0)=\infty, 
|U_{m_0+n}|
\leq A\sqrt{m_0+n} \text{ for all }n\geq 0\big)=0.
$$
\end{proposition}
\begin{proof}
Note first that since $\Re(\l)<\l_1<\alpha$ 
and  $\sigma'(m_0)\leq \sigma$ on $\{\sigma>m_0\}$, it follows from  Proposition~\ref{Un-subm} that  the process  $U_{n\wedge \sigma'(m_0)}$
is a submartingale.

It is easy to see that this submartingale has bounded increments. Indeed, 
recalling Remark~\ref{Un<C} and equation~\eqref{v-e} we have that
\begin{equation}
\label{MVcond1}
\Big| |U_{n+1}| - |U_n|\Big| \le 1,
\end{equation}
as $\zeta(n+1)=\zeta(n)\pm \be_j$ for some $j=1,2,\dots, N$.

Further, we are going to show that 
\begin{equation}
\label{MVcond2}
\E \left[\left(|U_{n+1}| - |U_n|\right)^2\| \F_n\right]\ge C\quad\text{on}\quad 
\zeta(n)\notin D_\psi\cup \{\zeta: |\zeta|\leq \gamma\},
% \{n<\min(\sigma, \nu)\},
\end{equation}
for a constant $C>0$. For that, clearly, it is enough to show the ``uniform ellipticity'' of $|U_n|$ when~$\zeta(n)$ is not in the ``bad set'': there are $\eps_0, h_0>0$ such that
\begin{equation}
\label{eq_unif_ell}
 \P_\zeta\big(\big| |U_{1}| - |U_0|\big|\geq \eps_0 \big) \geq h_0
\end{equation}
when $\zeta\notin D_\psi\cup \{\zeta: |\zeta|\leq \gamma\}$.

To prove the above, assume w.l.o.g.\ that
$$
\bv=(z_1,z_2,\dots,z_k,0,0,\ldots ,0),\quad 1\le k\le N,
\quad z_j\in \C\setminus\{0\}.
$$
Since multiplying the eigenvector by a constant $e^{i\varphi}$, $\varphi\in\R$, preserves its property of being an eigenvector and does not change $|U_n|$, we can assume w.l.o.g.\ that $z_1>0$ is real.
Consider the following two cases.

\paragraph{{\it Case 1:}} all other $z_j$s are real, which means that~$U_n$ also is. Assuming that $\zeta(n+1)=\zeta(n)\pm\be_1$, in one of the two cases $|U_{n+1}|=|U_n|+z_1$, which implies~\eqref{eq_unif_ell}.

\paragraph{{\it Case 2:}} some $z_j$ is complex, for example, $z_2=|z_2|e^{i\varphi_2}$,  where $|z_2|>0$, $\varphi_2\in (0,2\pi)\setminus\{\pi\}$.
\begin{figure}
\centering
 \includegraphics{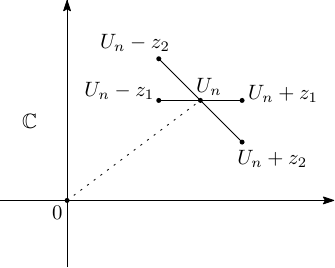}
 \caption{\small{On the proof of the uniform ellipticity in Case~2.}}
\label{f_unif_ell}
\end{figure}
In this case~$U_n$ may be complex, but the key observation is that~$z_1$ and~$z_2$ are not collinear, meaning that they ``move'' $U_n$ in different directions, as shown on Figure~\ref{f_unif_ell}. It is then 
easy to convince oneself that~\eqref{eq_unif_ell} holds in this case as well: 
if~$U_n$ is large and one of those two directions 
is (roughly) orthogonal to it, 
then moving~$U_n$ in that direction will not significantly change its absolute 
value;
however, moving~$U_n$ in the other direction
will do the job. Consequently, we have~\eqref{MVcond2}.

Now, consider a new process
\begin{equation}
\label{tilde-Un}
{\widetilde U}_n = 
 \begin{cases}
  U_{m_0+n}, & \text{ if }m_0+n<\sigma'(m_0) ,\\
  {\widetilde U}_{n-1}+1, & \text{ if }m_0+n\geq\sigma'(m_0).
 \end{cases}
\end{equation}
For this process,
we have $\big| |{\widetilde U}_{n+1}| - |{\widetilde U}_n|\big| \le 1$
and
$\E \big[\big(|{\widetilde U}_{n+1}| -
|{\widetilde U}_n|\big)^2\| 
\F_n\big]\ge C$ for \emph{all}~$n$; therefore, we can apply the law of iterated logarithms for sub-martingales, see e.g.~\cite[Lemma 2]{MV}, p.~947, with $a=0$ (which is, in turn, based on~\cite[Proposition (2.7)]{Freedman}), and this finishes the proof of Proposition~\ref{cla}.
\end{proof}

\begin{proposition}
\label{lemSPAs1}
Recall the process  $W_n$  defined in~\eqref{Wn}.
Given  $\eps>0$ the stopping time
$$\min\{n: W_n<\eps \}\wedge \sigma'(m_0)$$
is a.s.\ finite for any~$m_0$.
\end{proposition}
\begin{proof}
Recall the  process $\widetilde U_n$ defined in~\eqref{tilde-Un}. 
By Proposition~\ref{cla} we have that, with probability one,
$|\widetilde U_n|$ will  reach  the level of $A\sqrt{n}$  for any constant $A>0$. 
At the same time, it follows from Remark~\ref{Un<C} that 
$V_n\le n$, and, hence, 
$$W_n=\frac{V_n}{|U_n|^2}\le \frac{1}{A^2},$$ 
which can be made arbitrarily small by choosing $A$ sufficiently large.
\end{proof}

\begin{proposition}
\label{lemSPA4}
Let $C_1$ and $C_3$ be the constants from Proposition~\ref{Wn-supermartingale} 
and Proposition~\ref{lemSPA3} respectively.
Let $0<\eps<\min (C_1,C_3)$. If $W_{m_0}\le \eps^2$ then
$$
\P(\sigma'(m_0)=\infty\|\F_{m_0})\le \eps.
$$
\end{proposition}
\begin{proof}
Consider the stopping time 
$$
\tau(m_0)=\inf\{n>m_0:\ W_n>\eps\}.
$$

Since $\sigma'(m_0)\leq \sigma$ on $\{\sigma>m_0\}$, the process $W_{n\wedge \sigma'(m_0)}$ is a supermartingale on $\{W_n\le C_1\}$, by Proposition~\ref{Wn-supermartingale}. Applying the optional stopping theorem  to  the supermartingale 
$W_{n\wedge \sigma'(m_0)\wedge \tau(m_0)}$ gives 
that 
\begin{align}
\label{eqtaufin}
\P\big(\tau(m_0)<\min(\sigma'(m_0),\infty)\big)\le \frac{W_{m_0}}{\eps}
\leq \eps. 
\end{align}

Further, recall the process $H_n$ defined in~\eqref{Hn}.
It follows from Proposition~\ref{lemSPA3} and bound~\eqref{Vn>0} 
that 
\begin{equation}
\label{Hn-n}
\E(H_{n+1}-H_n\| \F_n)\le -\frac{C_6}{n+C_5}\quad\text{on}\quad 
\{\sigma'(m_0)>n,\, 
V_n\leq C_3|U_n|^2 %,\, |\zeta(n)|>\gamma
\},
\end{equation}
for some constants $C_5, C_6>0$ (and where  $C_3$ and $\gamma>0$ are  the constants from Proposition~\ref{lemSPA3}).

Denoting for short 
$$
\widetilde H_n:=H_{\min(n,\sigma'(m_0),\tau(m_0))}
$$
we can rewrite equation~\eqref{Hn-n} as follows
$$
\E(\widetilde H_{n+1}-\widetilde H_n\| \F_n)\le -\frac{C_6}{n+C_5}
1_{\{\tau(m_0)\wedge \sigma'(m_0)>n\}}.
$$
Taking the expectation and summing up from $m_0$ to $m_0+M$, we get, since 
$\widetilde H_n\ge 0$, that 
\begin{align*}
- \E \widetilde H_{m_0}\le \E \widetilde H_{m_0+M}-\E \widetilde H_{m_0}\le 
-\sum_{n=m_0}^{m_0+M}\frac{C_6}{n+C_5}\P(\min(\tau(m_0),\sigma'(m_0))>n).
\end{align*}
Now, if $\P(\min(\tau(m_0),\sigma'(m_0))>n)$ does not decay to zero
as $M\to\infty$, 
then the quantity in the right-hand side of the preceding display 
 converges to minus infinity, which is impossible. 
Hence
$\P(\min(\tau(m_0),\sigma'(m_0))>n)\downarrow 0$ implying that 
$\P(\min(\tau(m_0),\sigma'(m_0))=\infty)=0$. As 
a result, from~\eqref{eqtaufin}, we obtain
\begin{align*}
\P(\sigma'(m_0)=\infty)&=\P(\sigma'(m_0)=\infty,\tau(m_0)=\infty)
+\P(\sigma'(m_0)=\infty,\tau(m_0)<\infty)\\
&
\le 0+\P(\tau(m_0)<\infty)\le\eps.
\end{align*}
\end{proof}

Now the a.s.\ finiteness of~$\sigma'(\cdot)$
follows from Propositions~\ref{lemSPAs1}
and~\ref{lemSPA4}, since $\eps$ can be chosen arbitrary small.
Corollary~\ref{c_sigma_sigma'} then implies
 Lemma~\ref{sigma_finite}.

\section{Appendices}
\subsection{Appendix 1: Asymptotic behaviour of the total transition rate}
\label{rate_behaviour}
In this section we would like to present additional results on asymptotic behaviour of the total rate $R_n$ defined in~\eqref{Rn}.
Note that on $\{\zeta(n)>0\}$, equations~\eqref{Rn} and~\eqref{drift1} imply that 
\begin{equation}
\label{Rn-drift}
\begin{split}
\E\left(R_{n+1}-R_n|\F_n\right)=\frac{{\bf 1}\cdot(\alpha\I+\A)(\alpha \I-\A)\zeta(n)}{
{\bf 1}\cdot(\alpha\I+\A)\zeta(n)}&=\frac{{\bf 1}\cdot(\alpha\I-\A)(\alpha \I+\A)\zeta(n)}
{{\bf 1}(\alpha\I+\A)\zeta(n)}\\
&\leq \frac{\alpha{\bf 1}\cdot(\alpha\I+\A)\zeta(n)}{{\bf 1}\cdot(\alpha\I+\A)\zeta(n)}=\alpha,
\end{split} 
\end{equation}
where we used the fact that matrices $\alpha\I+\A$ and $\alpha \I-\A$ commute. 
The bound~\eqref{Rn-drift} implies that, while the process is away from the boundary, i.e. $\zeta(n)>0$, then, with high probability, $R_n\leq (\alpha+\eps)n$ for  sufficiently small $\eps>0$.  We skip the details of the corresponding proof. Instead, we are going to show  a more  subtle fact. Namely,  if $\l_1<\alpha$, then $R_n$ can be majorized by a random process which mean jump equals {\it exactly} $\alpha$. Precise statements are given in Propositions~\ref{lem-TR},~\ref{lem-bu} and~\ref{lem-LLN} below. 

\begin{proposition}
\label{lem-TR}
Suppose that $\l_1<\alpha$ and define
\begin{align*}
%\label{Tdef}
T_n=\alpha\bu^T\, \zeta(n)=\alpha \bu \cdot \zeta(n) ,\ \text{where }\ 
\bu=\left(\alpha\I+\A^\T\right) \left(\alpha\I-\A^\T\right)^{-1}{\bf 1}\in\R_+^N.
\end{align*}
Then $T_n\ge R_n$.
\end{proposition}
\begin{proof}[Proof of Proposition~\ref{lem-TR}]
Recall that $\l_1>0$ is the largest in absolute value eigenvalue of the matrix $\A$, and, hence, of the transposed matrix $\A^\T$. Since $\l_1<\alpha$ the matrix $\alpha\I-\A^{\T}$ is invertible. Therefore, both the vector $\bu$ and the quantity $T_n$ are properly defined. Further, observe that
$$
\left(\left[\I-\A^{\T}\alpha^{-1}\right]^{-1}\right)^{\T}=
\left(\left[\I-\A^{\T}\alpha^{-1}\right]^{\T}\right)^{-1}=(\I-\A\alpha^{-1})^{-1}
=\I+\sum_{k=1}^{\infty}\A^k\alpha^{-k}.
$$
Then we get the following
\begin{align*}
T_n-R_n&=\alpha{\bf 1}\cdot \left([\alpha\I-\A^{\T}]^{-1}\right)^{\T}\left(\alpha\I+
\A^{\T}\right)^{\T}\zeta(n)-{\bf 1}\cdot(\alpha\I+\A)\zeta(n)
\\
%    &={\bf 1}\cdot \left((\I-\A^{\T}\alpha^{-1})^{-1}\right)^{\T}\left(\alpha\I+\A\right)\zeta(n)-{\bf 1}\cdot (\alpha\I+\A\zeta(n)\\
&={\bf 1}\cdot\left(\I-\A\alpha^{-1}\right)^{-1}(\alpha\I+\A)\zeta(n)-{\bf 1}\cdot(\alpha\I+\A)\zeta(n)
\\
&=
{\bf 1}\cdot \left(\left(\I-\A\alpha^{-1}\right)^{-1}-\I \right)(\alpha\I+\A)\zeta(n)
\\
&={\bf 1}\cdot \left(\sum_{k=1}^{\infty}\A^k\alpha^{-k}\right)(\alpha\I+\A)\zeta(n) \ge 0,
\end{align*}
as claimed.
\end{proof}

Now we will show that $T_n$, roughly speaking, behaves like a random walk with a constant drift.
\begin{proposition}
\label{lem-bu}
Under assumptions of Proposition~\ref{lem-TR}
$$
\E(T_{n+1}-T_{n}\|\F_n)=\alpha \quad\text{on}\quad\{\zeta(n)>0\}.
$$
\end{proposition}
\begin{proof}[Proof of Proposition~\ref{lem-bu}]
In all equations in the proof of the proposition we assume that $\zeta(n)>0$. Then, equations~\eqref{Rn} and~\eqref{drift1} imply that
$$
\E(\zeta(n+1)-\zeta(n)\|\F_n)=\frac{(\alpha\I-\A) \zeta(n)}{{\bf 1}\cdot(\alpha\I+\A)\zeta(n)},
$$
therefore
$$
\E(T_{n+1}-T_{n}\|\F_n)=\frac{\alpha\bu\cdot (\alpha\I-\A) \zeta(n)}{{\bf 1}\cdot(\alpha\I+\A)\zeta(n)}.
$$
Observe thatt
\begin{align*}
\bu\cdot(\alpha\I-\A)&=\left[\left(\alpha\I+\A^{\T}\right)\left(\alpha\I-\A^{\T}\right)^{-1}{\bf 1}\right]^{\T} (\alpha\I-\A)\\
&={\bf 1}^\T(\alpha\I-\A)^{-1} (\alpha\I+\A)(\alpha\I-\A)\\
&={\bf 1}^\T(\alpha\I+\A)={\bf 1}\cdot(\alpha\I+\A),
\end{align*}
since $\alpha\I+\A$ and $\alpha\I-\A$ commute. Therefore,
$$
\E(T_{n+1}-T_{n}\|\F_n)=\frac{\alpha{\bf 1}\cdot(\alpha\I+\A)\zeta(n)}{{\bf 1}\cdot(\alpha\I+\A)\zeta(n)}=\alpha,
$$
as required. The proposition is proved.
\end{proof}

The next statement, which is a sort of a strong law of large numbers, is adapted from~\cite[Lemma 6]{ShV19}.
\begin{proposition}
\label{lem-LLN}
Under assumptions of Proposition~\ref{lem-TR}
$$
\lim_{n\to\infty} \left|\frac{T_n}{n}-\alpha\right| 1_{\{\sigma=\infty\}}= 0 \text{ a.s.},
$$
where $\sigma$ is the stopping time defined in~\eqref{sigma}.
\end{proposition}
\begin{proof}[Proof of Proposition~\ref{lem-LLN}] Let 
\begin{equation*}
%\label{Tn-hat}
\hat T_n=\begin{cases}
T_n-\a n&\text{if } n<\sigma,\\
T_\sigma-\a \sigma&\text{if } n\ge \sigma.
\end{cases}
\end{equation*}
Then $\hat T_n$ is a martingale with jumps bounded by some $c\in(0,\infty)$, since $\zeta(n)$ has uniformly bounded increments. Fix an $\eps>0$. By the Azuma-Hoeffding inequality, we have that 
\begin{equation*}%\label{Tn-Azuma}
\P(|\hat T_n-\hat T_0|\ge \eps n)\le 2\,e^{-\frac{n\, \eps^2}{2c^2}},
\end{equation*}
and by the Borel-Cantelli lemma the event above occurs finitely often. Since $\eps>0$ is arbitrary and $\hat T_0/n\to 0$ we get that $\lim_{n\to\infty} \hat T_n/n=0$ a.s. Next,
\begin{align*}
\left|T_n/n -\a \right|1_{\{\sigma=\infty\}}
=
\left|\hat T_n/n \right|1_{\{\sigma=\infty\}}
\le \left|\hat T_n/n \right| \to 0,
\end{align*}
finishing the proof.
\end{proof}

\subsection{Appendix 2: examples} \label{sec-examples}
In this section we provide some examples. Suppose that the interaction matrix is $\A=\beta\A_G$, where $\beta>0$ is a given constant and $\A_G$ is the adjacency matrix of a {\it non-directed} connected graph with $N\geq 2$ vertices and a constant vertex degree $d$. The latter means that each
 vertex is connected exactly to $d$ other vertices. In this case $d$ is the largest eigenvalue of the adjacency matrix $\A_G$ (i.e. the largest eigenvalue of the graph), so that $\l_1=\beta d$ is the largest eigenvalue of the interaction matrix $\A$. It is convenient to choose the corresponding eigenvector as follows 
$\bv_1={\bf 1}=\sum_{i=1}^N\be_i\in \Z_{+}^N$. 
Then the process $V_n$ defined in the general case by equation~\eqref{Sn} becomes
\begin{equation*}%\label{Sn1}
V_n=\zeta_1(n)+\cdots  +\zeta_N(n).
\end{equation*}
\begin{remark}
{\rm Note that in this special case 
 the process $V_n$ behaves as a simple random walk, that is  
\begin{equation*}
%\label{rw}
\begin{split}
\P(V_{n+1}-V_n=1|\zeta(n)>0)&=\frac{\alpha }{\alpha+\beta d}\\
\P(V_{n+1}-V_n=-1|\zeta(n)>0)&=\frac{\beta d}{\alpha+\beta d}
\end{split}
\end{equation*}
and the total rate $R_n$ (defined in~\eqref{Rn}) is proportional to $V_n$, that is 
$R_n=(\alpha+\beta d)V_n$ on $\{\sigma>n\}$.
}
\end{remark}
Now, let $G=(V, E)$ be a complete graph with $N$ vertices. This is a special case of a regular graph with the constant vertex degree $d=N-1$. It is easy to see that in this case the number of possible limit configurations is $N$. The corresponding interaction matrix has only two different eigenvalues, i.e. $\l_1=(N-1)\beta$ and $\l_2=-\beta$. The eigenvalue $\l_1$ is of multiplicity $1$, and the other eigenvalue $\l_2$ is of multiplicity $d = N-1$. The set of corresponding eigenvectors can be chosen as follows:
$$
\bv_1={\bf 1}\quad\text{and}\quad \bv_i = \be_1 - \be_ i ,\, i = 2, \ldots , N.
$$
If $\alpha>(N-1)\beta$, then, by Remark~\ref{linear-time}, the process $V_{n\wedge\sigma}$ 
is a non-negative supermartingale with a strictly negative mean jump, and, hence,  
the first extinction occurs in time linear in $V_0=\zeta_1(0)+\cdots  
+\zeta_N(0)$. If $\alpha<(N-1)\beta$, then any process $\bv_i\cdot \zeta(n)$, 
$i=2,\ldots ,N,$ can be used to construct the process $U_n$ (see~\eqref{Un}).

For example, if $N=2$, then we get a special case of the model studied in~\cite{ShV19}. 
In this particular case we have the following 
\begin{align*}
\lambda_1&=\beta,\, \bv_1=(1,1)^\T,
%\label{a1}
\\
\lambda_2&=-\beta, \, \bv_2=(1,-1)^\T, %\label{a2}
\\
V_n&=\bv_1\cdot\zeta(n)=\zeta_1(n)+\zeta_2(n),%\label{a3}
\\
U_n&=\bv_2\cdot\zeta(n)=\zeta_1(n)-\zeta_2(n)
%\label{a4}
.
\end{align*}

For an illustration, we present in Figure~\ref{complete11} a simulation of the DTMC $\zeta(n)$ in the case of the complete graph with $N=11$. The plot shows positions of the components of the process as functions of time. One can see that eventually only a single component survives. Similar simulations in the case of a complete graph with two vertices are shown in Figures~\ref{figB} and~\ref{figD}. Table~\ref{Table} provides a summary of the simulation results.

\begin{figure}[H]
\centering
 \includegraphics[scale=0.3]{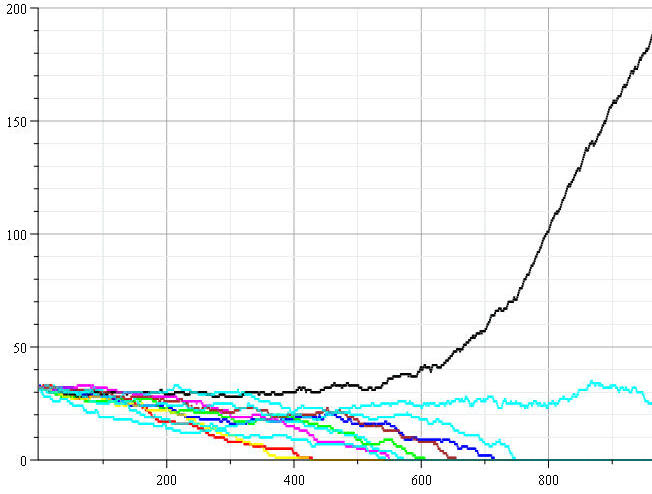}
 \vspace*{-4mm}
 \caption{\small{Simulation of the DTMC $\zeta(n)$ on a complete graph with $11$ vertices.}}
\label{complete11}
\end{figure}

\begin{figure}[H]
\centering
 \includegraphics[scale=0.2]{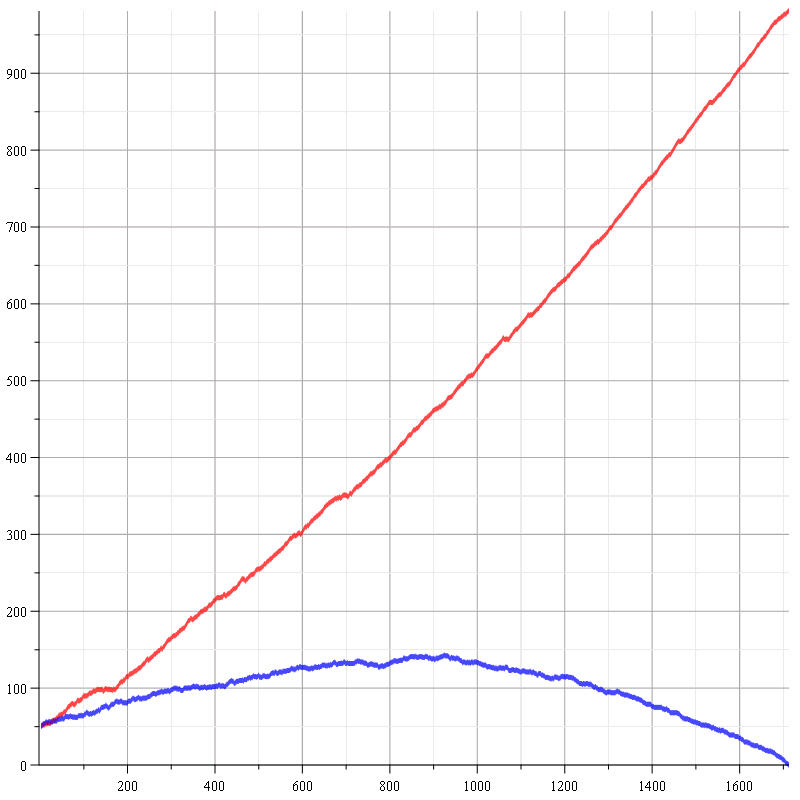}
\quad
 \includegraphics[scale=0.2]{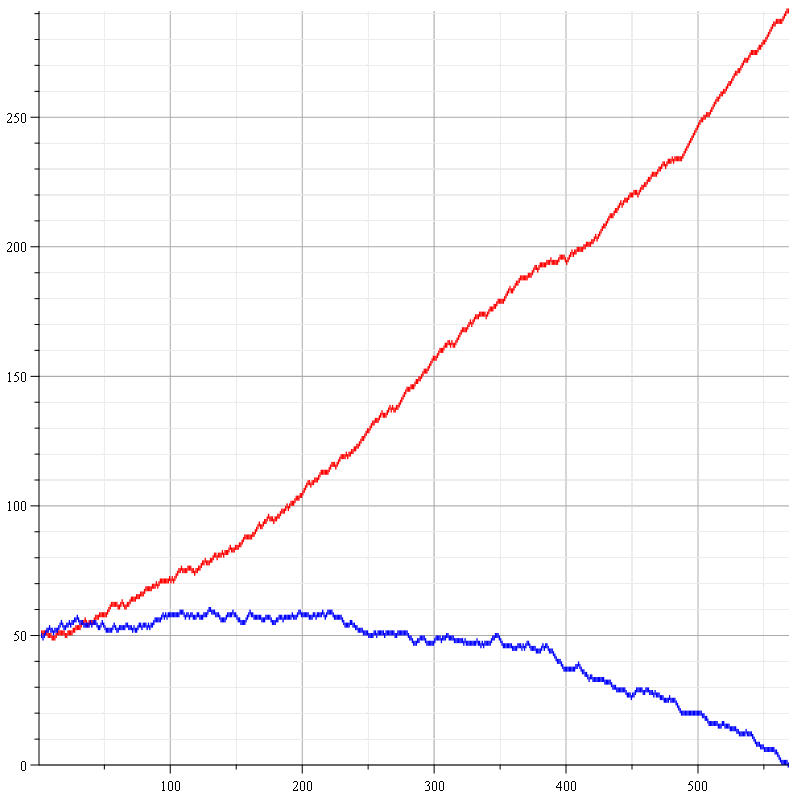}
 \caption{ \small{Simulation of lifetimes of the components of the DTMC $\zeta(n)$ on a complete graph with $2$ vertices. Parameters: $\alpha=1$, $\beta=0.3$ (left) and $\beta=0.5$ (right)}}
\label{figB}
\end{figure}

\begin{figure}[H]
\centering
 \includegraphics[scale=0.2]{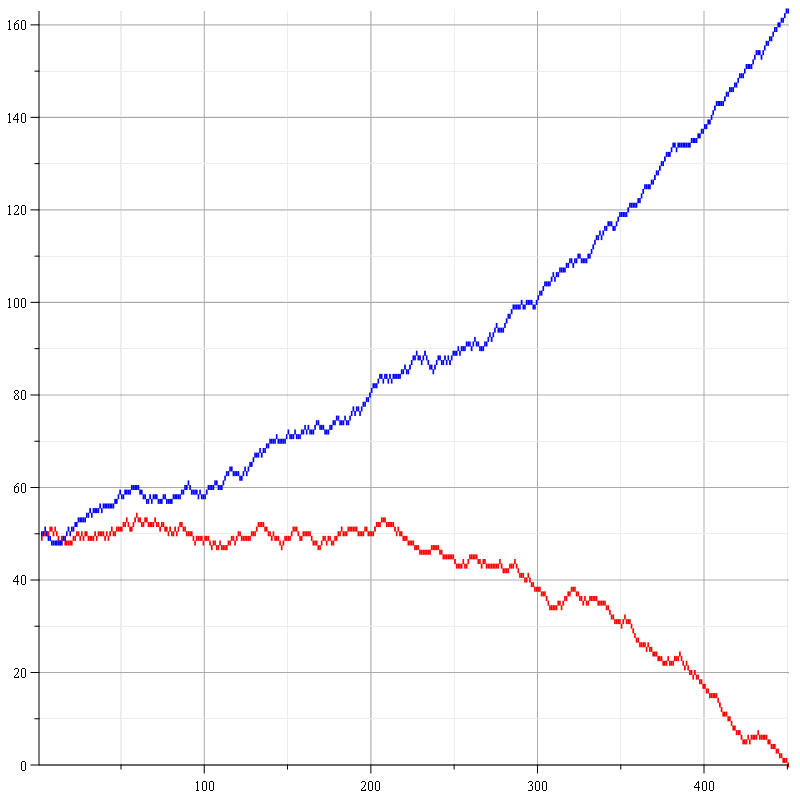}
\quad
 \includegraphics[scale=0.2]{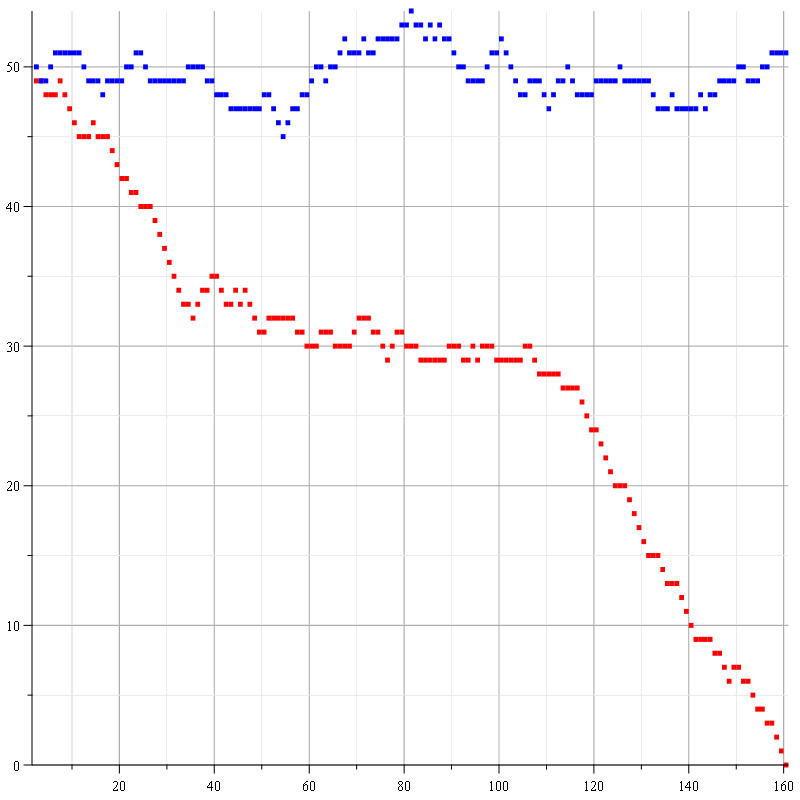}
 \caption{\small{The embedded process $\zeta(n)$ on a complete graph with $2$ vertices. $\alpha=1$, $\beta=0.7$ (left) and $\beta=1.5$ (right)}}
\label{figD}
\end{figure}

\begin{table}[H]
\centering 
\begin{tabular}{|c|c|c|c|c|}
\hline
 & $\beta=0.3$ & $\beta=0.5$ & $\beta=0.7$ & $\beta=1.5$\\
\hline
$\alpha=1$ & $\sigma=1713$ & $\sigma=570$ & $\sigma=500$ & $\sigma=160$\\
\hline
\end{tabular}
\caption{ Sample extinction times $\sigma$, $\zeta(0)=(50, 50)$ }
\label{Table}
\centering
\end{table}

Table~\ref{Table2} gives numbers of possible limit configurations, the Perron-Frobenius eigenvalue $\l_1$ and a variant of the eigenvalue $\l_N$ for the cycle, line and star graphs (with $N\geq 2$ vertices). All graphs are non-directed, and 
$v\sim u$ denotes that vertices $v$ and $u$ are connected by an edge.

\begin{table}[H]
\centering 
 \begin{tabular}{|c|c|c|c|}
\hline
 & & \\
 Cycle graph & Line graph & Star graph \\
 $1\sim 2\sim\cdots  \sim N\sim 1$& $1\sim 2\sim\cdots  \sim N-1\sim N$& $1\sim 
i$, $i=2,\ldots ,N-1$\\
 & & 
($1$ is the central vertex)\\
\hline
 & & \\
$M(G_N)=F_{N-1}+F_{N-3}-1$ & $M(G_N)=F_N-1$ & $M(G_N)=2^{N-1}$ \\
 & & \\
\hline
 & & \\
 $\l_1=2\beta$ & $\l_1=2\beta\cos\left(\frac{\pi }{N+1}\right)$ & $\l_1=\beta\sqrt{N-1}$\\
 & & \\
\hline
 & & \\
$\l_N=-2\beta \cos\left(\frac{\pi \, {\bf 1}_{\{N\text{ is odd}\}}}{N}\right)$ & $\l_N=-2\beta\cos\left(\frac{\pi }{N+1}\right)$ & $\l_N=-\beta\sqrt{N-1}$\\
 & & \\
\hline
 \end{tabular}
 \caption{ {\small $M(G_N)$ denotes the number of the possible limit configurations for a graph $G_N$. $F_N$ is the~$N$-th Fibonacci number.}}
\label{Table2}
\centering
\end{table}

\subsection{Appendix 3: conjecture for the model with immigration}
First of all, note that motivation for the current paper comes from~\cite{ShV19}, where we considered a similar model only in the case where $N=2$. In the model of~\cite{ShV19} we allowed ``immigration'', i.e.\ $q_{\bx,\bx+\be_i}=\alpha_i x_i+\l_i$, where $\l_i>0$ is the immigration rate, and we also allowed $\alpha_i$ to be different. On the other hand, in~\cite{ShV19} we demanded that $a_{12}>0$ and $a_{21}>0$, which ensured that matrix $\A$ is irreducible. In fact, possible non-reducibility of $\A$ (and, hence, non-connectedness of $G$) causes a substantial challenge in our current model as we have to deal with multiple possibilities for the structure of the underlying graph, and use the recursion in the proof. 

Including the immigration rate into the current model with arbitrary $N$ is straightforward. However, some computations will become more tedious, and we chose not to do so. At the same time, we believe it is possible to extend the results of Theorem~\ref{T1} of the current paper and \cite[Theorem~2]{ShV19} as follows.

\begin{conjecture}
Suppose that we are given the interaction matrix satisfying Definition~\ref{D1}, and the transition rates are given by~\eqref{rates} with the correction that $q_{\bx\by}=\alpha x_i+\l_i$, $\by=\bx+\be_i$, $\l_i\ge 0$. Then there exists a.s.\ a time $T$ and a subset of vertices ${\cal I}=\{i_1,i_2,\dots, i_K\}\subset V$ satisfying the conditions of Theorem~\ref{T1}, such that for all $t\ge T$
$$
X_i(t)\to\infty \text{ if and only if }i\in {\cal I}.
$$
Moreover, for each $j$ such that $\l_j>0$, either $j\in{\cal I}$, or $j\notin{\cal I}$ but there is an $i\in{\cal I}$ such that $i\dir j$. Finally, for all $j\notin {\cal I}$
$$
\liminf_{t\to\infty} X_j(t)=0 \text{ and } \limsup_{t\to\infty} X_j(t)=1 \quad \text{a.s.}
$$
\end{conjecture}

\section*{Acknowledgement}
%We would  like to thank the anonymous referees for their comments and identifying a mistake in the original proof of Lemma~\ref{sigma_finite}.
We are 
%also 
 grateful to E.~Crane, A.~Holroyd, and J.~F.~C.~Kingman for useful remarks on the paper.

\end{document}